\documentclass[11pt]{amsart}
\usepackage[margin=2.75cm]{geometry}

\usepackage[T1]{fontenc}
\usepackage{lmodern}
\usepackage{amssymb,amsmath,amsthm,mathrsfs}
\usepackage{microtype}
\usepackage{enumitem}
\usepackage{needspace}
\usepackage{xcolor}
\usepackage[colorlinks,allcolors=blue]{hyperref}
\definecolor{gpblue}{RGB}{0,0,255}
\definecolor{gpOrange}{RGB}{226,107,0}
\definecolor{gpGreen}{RGB}{0,127,70}
\definecolor{gpRed}{RGB}{207,49,48}
\definecolor{gpViolet}{RGB}{136,63,152}
\definecolor{gpGold}{RGB}{165,139,24}

\allowdisplaybreaks[2]
\newtheorem{theorem}{Theorem}[section]
\newtheorem{proposition}[theorem]{Proposition}
\newtheorem{lemma}[theorem]{Lemma}
\newtheorem{corollary}[theorem]{Corollary}
\theoremstyle{definition}
\newtheorem{definition}[theorem]{Definition}
\newtheorem{example}[theorem]{Example}
\newtheorem{application}[theorem]{Application}
\theoremstyle{remark}
\newtheorem{remark}[theorem]{Remark}

\newcommand{\C}{\mathcal C}
\newcommand{\one}{\mathbf 1}
\newcommand{\CC}{\mathbb C}
\newcommand{\kk}{\Bbbk}
\newcommand{\Exch}{\mathfrak E}
\DeclareMathOperator{\Hom}{Hom}
\DeclareMathOperator{\End}{End}
\DeclareMathOperator{\Aut}{Aut}
\DeclareMathOperator{\Alg}{Alg}
\DeclareMathOperator{\Irr}{Irr}
\DeclareMathOperator{\Corep}{Corep}
\DeclareMathOperator{\cdeg}{cdeg}
\DeclareMathOperator{\HH}{HH}
\DeclareMathOperator{\Rep}{Rep}
\DeclareMathOperator{\id}{id}
\DeclareMathOperator{\Vect}{Vec}
\DeclareMathOperator{\ev}{ev}
\DeclareMathOperator{\coev}{coev}

\title[Local rigidity for separable algebra objects]
{Local Ocneanu rigidity for separable algebra objects}

\author{Tinhinane Amina Azzouz}
\address{T.A. Azzouz, Beijing Institute of Mathematical Sciences and Applications, Huairou District, Beijing, China}
\email{azzouzta@bimsa.cn}

\author{Mainak Ghosh}
\address{M. Ghosh, Hetao Institute of Mathematics and Interdisciplinary Sciences, Futian District, Shenzhen, China}
\email{ghosh.main@gmail.com}

\author{S\'ebastien Palcoux}
\address{S. Palcoux, Beijing Institute of Mathematical Sciences and Applications, Huairou District, Beijing, China}
\email{sebastienpalcoux@gmail.com}
\urladdr{https://sites.google.com/view/sebastienpalcoux}

\newcommand{\InternalRoadmapAddon}{}

\keywords{Local Ocneanu rigidity; separable algebra object; Hochschild
cohomology; affine variety; Zariski tangent space; B\'ezout bound;
Frobenius subalgebra; left coideal subalgebra; intermediate operator algebra}
\subjclass[2020]{Primary 18M05, 46L37;
Secondary 14B10, 14R20, 16E40, 16T05}

\hypersetup{
  pdftitle={Local Ocneanu rigidity for separable algebra objects},
  pdfauthor={Tinhinane Amina Azzouz, Mainak Ghosh, and S\'ebastien Palcoux},
  pdfsubject={Local rigidity and effective bounds for separable algebra objects},
  pdfkeywords={local Ocneanu rigidity, separable algebra object, Hochschild cohomology, affine variety, Zariski tangent space, Bezout bound, Frobenius subalgebra, left coideal subalgebra, intermediate operator algebra}
}

\begin{document}
\begin{abstract}
We establish local Ocneanu rigidity for separable algebra objects in
Hom-finite monoidal categories $\C$ over an algebraically closed field $\kk$.
A separability morphism contracts the first two Hochschild cohomology
groups without requiring an abelian ambient category.  Separable algebra
structures and homomorphisms from separable sources have open algebraic-group
orbits, and affine B\'ezout estimates give effective finiteness results.
For Frobenius subalgebras, exchange relations replace the source and
embedding data by a single self-dual idempotent.  We prove that their
separable inner-conjugacy classes are open in the exchange locus,
yielding a bound of $2^{\dim_{\kk}\End_{\C}(X)}$ for any ambient Frobenius algebra $X$;
connectedness gives the same bound on the actual number of subalgebras.
As an application, we obtain an effective form of the
Etingof--Walton finiteness theorem: every finite-dimensional semisimple
Hopf algebra $H$ over $\CC$ has fewer than $2^{\dim_{\CC}H}$ left coideal
subalgebras.  In the unitary setting, we bound $E$-compatible intermediates of
$C^*$-algebra and von Neumann algebra inclusions up to unitary conjugacy,
under finite-index and finite-center hypotheses.
For irreducible finite-index subfactors $N\subset M$, we
improve the $9^{[M:N]}$ bound of Bakshi--Das--Liu--Ren to fewer than
$2^{[M:N]}$ intermediate subfactors.
\end{abstract}

\maketitle

\setcounter{tocdepth}{1}
\tableofcontents

\section{Introduction}

Ocneanu rigidity is a global finiteness principle for fusion categories,
tensor functors, and semisimple module categories: once the relevant discrete
data are fixed, the corresponding structures admit no nontrivial
infinitesimal deformations.  Following the historical account in Etingof--Nikshych--Ostrik
\cite{ENO}, the unitary argument was suggested by Ocneanu and was
subsequently developed algebraically by Blanchard and Wassermann;
that paper gives a general published proof for fusion categories; for
semisimple module categories with a prescribed number of simple isomorphism
classes in characteristic zero,
see \cite[Corollary~2.35 and \S{}7]{ENO}.  We use
the terminology of \cite{EGNO}.  Related deformation complexes were
introduced by Crane--Yetter \cite{CraneYetter} and Davydov
\cite{Davydov}.

This paper studies a local counterpart.  The ambient monoidal category and
the underlying object are fixed, while an algebra structure, an algebra map,
or a subalgebra embedding is allowed to vary.  Our purpose is to isolate the
algebraic mechanism, determine the exact
hypotheses, make the conclusion effective, and connect it with Frobenius
subalgebras, rational representations of linearly reductive groups,
left coideal subalgebras of semisimple Hopf algebras, and intermediate
operator algebras.

The deformation-theoretic core is elementary and categorical.  If $A$ is a
separable algebra object and $M$ is an $A$--$A$-bimodule, a separability
morphism defines contractions in Hochschild degrees one and two.
No abelian structure is needed, and
\[
  \HH^1(A,M)=\HH^2(A,M)=0.
\]
The two vanishings serve different purposes.  The equality
$\HH^2(A,A)=0$ identifies infinitesimal deformations of the multiplication
with infinitesimal changes of coordinates, whereas
$\HH^1(A,{}_fX_f)=0$ identifies infinitesimal deformations of an algebra map
$f:A\to X$ with infinitesimal inner conjugations.  The deformation
viewpoint goes back to Gerstenhaber \cite{Gerstenhaber} and, for algebra
morphisms, Gerstenhaber--Schack \cite{GerstenhaberSchack83}.
Ostrik \cite{Ostrik} already sketches the
Hochschild-cohomological route to rigidity for categories arising from
conformal field theory.
Here the low-degree contractions are explicit without an abelian ambient
category (Proposition~\ref{prop:hochschild-vanishing}); for the abelian
monoidal setting, see Ardizzoni--Menini--\c{S}tefan \cite{AMS}.

Our first main result is qualitative.

\begin{theorem}\label{thm:intro-local}
Let $\kk$ be an algebraically closed field and let $\C$ be a $\kk$-linear
monoidal category with finite-dimensional
morphism spaces.
\begin{enumerate}
  \item Each fixed object $B$ in $\C$ supports only finitely many isomorphism
  classes of separable unital algebra structures
  \textup{(Theorem~\ref{thm:fixed-object})}.
  \item For separable $A$ and any algebra object $X$, the unital algebra maps
  $A\to X$ form finitely many inner-conjugacy orbits.  If $X$ is connected,
  the map set itself is finite
  \textup{(Theorem~\ref{thm:homomorphisms})}.
  \item Any family of separable algebra subobjects of $X$ with finitely
  many isomorphism classes of underlying source objects has finitely many
  inner-conjugacy classes and, when $X$ is connected, finitely many
  equivalence classes over $X$.  In particular, this applies to all separable split
  algebra subobjects
  \textup{(Theorem~\ref{thm:subobjects} and
  Corollary~\ref{cor:split-subobjects})}.
\end{enumerate}
\end{theorem}

The third assertion separates local rigidity from finiteness of the
possible source objects.  This finiteness holds for all subobjects in a
semisimple category (Definition~\ref{def:semisimple-linear-category}),
and may hold for particular objects in a nonsemisimple category
(Example~\ref{ex:strict-subobject}).  For split embeddings it follows from
Krull--Schmidt in the additive Karoubi completion
(Lemma~\ref{lem:split-source-types}); neither abelianity nor finite length
is needed.  Split separable subalgebras also admit an intrinsic exponential
bound without Frobenius structure (Remark~\ref{rem:split-quadratic-bound}).

Thus a connected algebra in a semisimple tensor category has finitely
many separable subalgebras, whereas a connected Frobenius algebra in any
tensor category has finitely many separable Frobenius subalgebras.
The difference is that Frobenius-subalgebra embeddings split by
definition; separability alone need not split an embedding.
Here tensor categories are understood in the sense of \cite{EGNO},
which includes Hom-finiteness.  In particular, if $G$ is a linearly
reductive algebraic group (including reductive groups in characteristic
zero), a finite-dimensional rational $G$-algebra $X$ with $X^G=\kk1$
has finitely many $G$-stable separable unital subalgebras.

Finite-dimensional parameter spaces also yield explicit estimates.  For
an object $B$, set
\[
\begin{aligned}
 a_B&=\dim_{\kk}\Hom_{\C}(B\otimes B,B),&
 u_B&=\dim_{\kk}\Hom_{\C}(\one,B),\\
 t_B&=\dim_{\kk}\Hom_{\C}(B^{\otimes3},B),&
 E_B&=\dim_{\kk}\End_{\C}(B).
\end{aligned}
\]
After bases are chosen, associativity and unitality become polynomial
equations of total degree at most two in the coordinates of the
multiplication and unit.  Moreover, the change-of-coordinates orbit
representing each separable isomorphism class is open.  An affine B\'ezout
inequality therefore gives the following bound.

\begin{theorem}[Effective local rigidity]\label{thm:intro-effective}
The number of isomorphism classes of separable unital algebra structures on
$B$ is at most
\[
  2^{\min\{a_B+u_B,\,t_B+2E_B\}}.
\]
This is proved in Theorem~\ref{thm:fixed-object-bound}.
If $A$ is separable and $X$ is an algebra object, set
\[
\begin{aligned}
 b_{A,X}&=\dim_{\kk}\ker\bigl(\Hom_{\C}(A,X)\to
 \Hom_{\C}(\one,X),\ f\mapsto f\circ e_A\bigr),\\
 c_{A,X}&=\dim_{\kk}\Hom_{\C}(A\otimes A,X).
\end{aligned}
\]
Then the number of inner-conjugacy classes of unital algebra homomorphisms
$A\to X$ is at most $2^{\min\{b_{A,X},c_{A,X}\}}$.  For connected $X$,
this bounds the homomorphism set itself
\textup{(Theorem~\ref{thm:hom-bound})}.  The explicit resulting sums for
algebra subobjects are given in Theorem~\ref{thm:subobject-bound}, with a
bound depending only on the ambient algebra in
Corollary~\ref{cor:target-only-bound}.
\end{theorem}

Frobenius self-duality permits a smaller parameter space: it records
both the source algebra and its embedding by an endomorphism of the
ambient Frobenius algebra.  The projection approach to intermediate
subfactors goes back to Bisch \cite{Bisch}, and its exchange-relation
characterization appears in \cite{BischJones}.
The categorical exchange-relations theorem of Ghosh--Palcoux \cite[Theorem~4.1]{GhoshPalcouxExchange} encodes a
Frobenius subalgebra by a self-dual unital idempotent satisfying two
quadratic equations.  Our additional observation is that the joint
deformation of a separable source algebra and its embedding has an open
orbit.  Its inverse image in the exchange locus makes each separable
inner-conjugacy class open, even though ambient inner conjugation need
not preserve the chosen Frobenius counit.  Qualitative finiteness already follows from the split-embedding argument.
Counting these disjoint open sets gives the following bound directly in
terms of the ambient endomorphism space.

\begin{theorem}[Quadratic exchange bound]\label{thm:intro-exchange}
Let $\C$ be a Hom-finite $\kk$-linear monoidal category and let $X$
be a Frobenius algebra in $\C$.  Put
\[
 r_X:=\dim_{\kk}\{h\in\End_{\C}(X):h^*=h,\ h\circ e_X=0\},
\]
where $*$ denotes the Frobenius dual.  The number of inner-conjugacy
classes of separable Frobenius subalgebras, with conjugacy referring to
the underlying algebra subobjects, is at most
\[
 2^{r_X}\leq2^{\dim_{\kk}\End_{\C}(X)}.
\]
For $X\ne0$, the last exponent improves to
$\dim_{\kk}\End_{\C}(X)-1$.
If $X$ is connected, this bounds the number of equivalence classes over
$X$.  This is Theorem~\ref{thm:exchange-bound}.
\end{theorem}

We first apply the exchange bound to left coideal subalgebras.  The regular
comodule of a
finite-dimensional semisimple Hopf algebra is a connected Frobenius algebra,
and the Frobenius-subalgebra correspondence of Ghosh--Palcoux identifies its
Frobenius subalgebras with left coideal subalgebras
\cite[Theorem~11.11]{GhoshPalcouxLattices}.  After proving the required
categorical separability, we obtain an effective refinement of the theorem
of Etingof--Walton \cite[Theorem~3.6]{EtingofWalton}.

\begin{corollary}[Semisimple Hopf algebras]\label{cor:intro-hopf}
A semisimple Hopf algebra $H$ over $\CC$ with $\dim_{\CC}H=N$ has at most
$2^{N-1}$ left coideal subalgebras
\textup{(Corollary~\ref{cor:hopf-bound})}.
\end{corollary}

For a finite-index inclusion $N\subseteq M$ of $\mathrm{II}_1$ factors,
the exchange bound controls intermediate von Neumann algebras up to
conjugation by unitaries in $N'\cap M$.  Under irreducibility, these
classes are singletons, and the dimension comparison
$\dim_{\CC}(N'\cap M_1)\leq[M:N]$
in \eqref{eq:relative-commutant-index} gives the following consequence.

\begin{corollary}[Intermediate subfactors]\label{cor:intro-subfactor}
For every finite-index irreducible $\mathrm{II}_1$ subfactor $N\subset M$,
\[
 |\mathcal L(N\subset M)|
 \leq 2^{\dim_{\CC}(N'\cap M_1)-1}
 \leq 2^{\lfloor[M:N]\rfloor-1}.
\]
The estimate is proved in Theorem~\ref{thm:subfactor-bound}.
\end{corollary}

Qualitative finiteness was established by Watatani \cite{WatataniLattices}
and, for type~III subfactors, Teruya--Watatani \cite{TeruyaWatatani};
Longo \cite{Longo2003} also obtained explicit bounds.
The estimate above improves the whole-lattice bound $9^{[M:N]}$ of
Bakshi--Das--Liu--Ren \cite[Theorem~4.6]{BakshiDasLiuRen}.

\S{}\ref{sec:subfactors} also treats unital $C^*$-inclusions with a
conditional expectation of finite Watatani index and a finite-dimensional
center for the smaller algebra.  Analytic local unitary-conjugacy results
were obtained by Ino--Watatani \cite{InoWatatani} and Dickson \cite{Dickson}.
Gupta--Kumar \cite{GuptaKumar} proved qualitative finiteness for
$E$-compatible intermediates of irreducible finite-index unital inclusions
without simplicity assumptions.  In the simple case, quantitative bounds
were developed by Bakshi--Gupta \cite{BakshiGupta} and improved to $9^J$,
with $J$ the least Watatani index of a conditional expectation, by
Bakshi--Guin--Jana \cite{BakshiGuinJana}.
Our exchange-locus argument gives effective bounds up to relative-commutant
unitary conjugacy without irreducibility, and actual counts under
irreducibility.  The bounds also apply to compatible intermediate von
Neumann algebras under the same finite-index and finite-center hypotheses;
the separate subfactor statements use the Jones index
(Remark~\ref{rem:von-neumann-inclusions}).
An application to equivalence classes of states under bimodule quantum
channels, due to Huang--Jiang--Liu--Wu \cite{HuangJiangLiuWuStates},
is recorded in \S{}\ref{sec:examples}.

The quantitative estimates imply qualitative finiteness.  Their common
input is the low-degree separability contraction: degree two gives open
orbits for algebra structures, while degree one gives open
inner-conjugacy orbits for homomorphisms.  Combining them for a varying
source and embedding proves openness of the separable inner-conjugacy
classes in the exchange locus.  In the connected case these classes are
single points.  Affine degree then bounds their number.  We retain the
qualitative statements to expose the local mechanism and its precise
hypotheses.

The paper is organized as follows.  \S{}\ref{sec:setting} fixes the
terminology\InternalRoadmapAddon, and \S{}\ref{sec:hoch} proves the
low-degree Hochschild contraction.  \S{}\ref{sec:geometry} records the
algebraic-geometric tools.  \S\S{}\ref{sec:structures}--\ref{sec:subobjects}
treat algebra structures, homomorphisms, and subobjects.
\S{}\ref{sec:frobenius} develops the exchange locus and the general
Frobenius bound.  Hopf algebras and finite groups are treated in
\S{}\ref{sec:hopf}, followed by unitary tensor categories and
operator-algebra inclusions in \S{}\ref{sec:subfactors}, including the
comparison of intermediate-subfactor bounds in
\S{}\ref{subsec:subfactor-comparison}.  Examples and applications appear
in \S{}\ref{sec:examples}.  The scope, limitations, and open questions
are discussed in \S{}\ref{sec:scope}, including the comparison with
earlier categorical finiteness and lattice results in
\S{}\ref{subsec:comparison-methods}.
\section{Setting and terminology}\label{sec:setting}

\subsection{Categorical conventions}

Let $\kk$ be an algebraically closed field.  We use the terminology of
\cite{EGNO} whenever tensor categories are involved.  For the local
statements, however, $\C$ is only a $\kk$-linear monoidal category whose
composition and tensor product are $\kk$-bilinear and whose morphism spaces
are finite-dimensional.

We suppress associators and unitors, interpreting all tensor products and
composites by monoidal coherence.  No braiding is assumed, and the order of
the tensor factors is preserved throughout.

An \emph{algebra object} $A=(A,m_A,e_A)$ consists of morphisms
$m_A:A\otimes A\to A$ and $e_A:\one\to A$ such that
\[
 m_A\circ(m_A\otimes\id_A)=m_A\circ(\id_A\otimes m_A),
 \qquad
 m_A\circ(e_A\otimes\id_A)=\id_A
 =m_A\circ(\id_A\otimes e_A).
\]
A \emph{unital algebra homomorphism} $f:A\to X$ satisfies
\[
 f\circ e_A=e_X,
 \qquad
 f\circ m_A=m_X\circ(f\otimes f).
\]
An \emph{algebra subobject}, or \emph{subalgebra}, of $X$ is a pair $(A,i)$
consisting of an algebra object $A$ and a unital algebra monomorphism
$i:A\to X$.
We call $A$ the \emph{source} and $X$ the \emph{ambient algebra}; for an
arbitrary algebra map $A\to X$, $X$ is also its \emph{target}.
The subalgebra is \emph{split} if there is a morphism $p:X\to A$ in $\C$
with $p\circ i=\id_A$; $p$ is not required to be an algebra map.  The pairs $(A,i)$ and
$(A',i')$ are \emph{equivalent over $X$} if there is an algebra isomorphism
$h:A\to A'$ such that $i=i'\circ h$.

The algebra object $X$ is \emph{connected} if
$\dim_{\kk}\Hom_{\C}(\one,X)=1$.  Its unit morphism is nonzero: if $e_X=0$,
the left-unit identity gives $\id_X=0$, and then every morphism $\one\to X$
is zero.  Consequently $\Hom_{\C}(\one,X)=\kk e_X$.

\subsection{The morphism algebra and inner conjugacy}

The finite-dimensional vector space
$\Gamma(X):=\Hom_{\C}(\one,X)$ is an ordinary unital associative
$\kk$-algebra, with product
$u\star v:=m_X\circ(u\otimes v)$ and unit $e_X$.  Write
$\Gamma(X)^\times$ for its group of invertibles $u:\one\to X$, i.e. those
for which there is $u^{-1}:\one\to X$ satisfying
$u\star u^{-1}=e_X=u^{-1}\star u$.
For such $u$, define
\begin{equation}\label{eq:categorical-conjugation}
 c_u:=m_X\circ(m_X\otimes\id_X)
 \circ(u\otimes\id_X\otimes u^{-1}):X\longrightarrow X.
\end{equation}
Associativity, unitality, and the inverse identities give
\[
 c_u\circ e_X=e_X,\qquad
 c_u\circ m_X=m_X\circ(c_u\otimes c_u),\qquad
 c_{u^{-1}}\circ c_u=\id_X=c_u\circ c_{u^{-1}}.
\]
The same identities give
$c_{e_X}=\id_X$ and $c_{u\star v}=c_u\circ c_v$.
Thus $c_u$ is an algebra automorphism, and $\Gamma(X)^\times$ acts on
algebra homomorphisms $f:A\to X$ by
\begin{equation}\label{eq:conjugation-action}
 u\boldsymbol\cdot f:=c_u\circ f.
\end{equation}
We call the resulting equivalence relation \emph{inner conjugacy}.  If $X$
is connected, $\Gamma(X)=\kk e_X$, so the action is trivial.

Two algebra subobjects $(A,i)$ and $(A',i')$ are \emph{inner conjugate} if
there are $u\in\Gamma(X)^\times$ and an algebra isomorphism $h:A\to A'$ such
that $c_u\circ i=i'\circ h$.  For connected $X$, this is exactly equivalence
over $X$.

\subsection{Bimodules and separability}\label{subsec:bimodule-separability}

An $A$--$A$-bimodule is an object $M$ equipped with morphisms
$\ell_M:A\otimes M\to M$ and $r_M:M\otimes A\to M$ satisfying
\begin{align*}
 \ell_M\circ(m_A\otimes\id_M)
   &=\ell_M\circ(\id_A\otimes\ell_M),&
 \ell_M\circ(e_A\otimes\id_M)&=\id_M,\\
 r_M\circ(r_M\otimes\id_A)
   &=r_M\circ(\id_M\otimes m_A),&
 r_M\circ(\id_M\otimes e_A)&=\id_M,\\
 r_M\circ(\ell_M\otimes\id_A)
   &=\ell_M\circ(\id_A\otimes r_M).
\end{align*}
A bimodule morphism $q:M\to N$ satisfies
\[
 q\circ\ell_M=\ell_N\circ(\id_A\otimes q),
 \qquad
 q\circ r_M=r_N\circ(q\otimes\id_A).
\]
The regular bimodule on $A$ has both actions $m_A$.  On $A\otimes A$ we use
the outer regular actions $m_A\otimes\id_A$ and
$\id_A\otimes m_A$.

An algebra object $A$ is \emph{separable} if $m_A$ has an
$A$--$A$-bimodule section $s_A:A\to A\otimes A$.
The following equivalent formulation will be used in the Hochschild
contraction.  It holds in every monoidal category; linearity,
Hom-finiteness, and algebraic closedness are not needed.

\begin{lemma}[Separability morphism]\label{lem:separability-morphism}
For an algebra object $A$, the following are equivalent.
\begin{enumerate}[label=\textup{(\roman*)}]
 \item There is an $A$--$A$-bimodule morphism
 $s_A:A\to A\otimes A$ such that $m_A\circ s_A=\id_A$.
 Its associated separability morphism is
 \begin{equation}\label{eq:separability-morphism}
  p_A:=s_A\circ e_A:\one\longrightarrow A\otimes A.
 \end{equation}
 \item There is a morphism $p_A:\one\to A\otimes A$ such that
 \begin{equation}\label{eq:separability-relations}
  m_A\circ p_A=e_A,
  \qquad
  (m_A\otimes\id_A)\circ(\id_A\otimes p_A)
  =(\id_A\otimes m_A)\circ(p_A\otimes\id_A).
 \end{equation}
\end{enumerate}
Conversely, the section corresponding to $p_A$ is
\begin{equation}\label{eq:s-from-p}
 s_A=(m_A\otimes\id_A)\circ(\id_A\otimes p_A)
     =(\id_A\otimes m_A)\circ(p_A\otimes\id_A).
\end{equation}
\end{lemma}

\begin{proof}
Given $s_A$, put $p_A=s_A\circ e_A$.  The section identity gives
$m_A\circ p_A=e_A$.  Precomposing the left and right $A$-linearity
identities for $s_A$ with $\id_A\otimes e_A$ and
$e_A\otimes\id_A$, respectively, shows that both composites in
\eqref{eq:separability-relations} equal $s_A$.

Conversely, let $p_A$ satisfy \eqref{eq:separability-relations} and define
$s_A$ by either of the equal composites in
\eqref{eq:s-from-p}.  Associativity proves left and right
$A$-linearity.  Moreover,
\[
 m_A\circ s_A
 =m_A\circ(\id_A\otimes m_A)\circ(\id_A\otimes p_A)
 =m_A\circ(\id_A\otimes e_A)=\id_A.
\]
Thus $s_A$ is a bimodule section; the unit identities also give
$s_A\circ e_A=p_A$.
\end{proof}

\subsection{Underlying objects of subalgebras and split embeddings}

Several statements below do not require an abelian category, so we record
the precise additive terminology used in the paper.

\begin{definition}\label{def:semisimple-linear-category}
A $\kk$-linear category is \emph{Hom-finite} if every morphism space is
finite-dimensional, and \emph{Karoubian} if every idempotent splits.  An
additive $\kk$-linear category is \emph{Schur} if every indecomposable object
$S$ satisfies $\End_{\C}(S)=\kk$, and
$\Hom_{\C}(S,T)=0$ for nonisomorphic indecomposable objects $S,T$.  In this
paper, a \emph{semisimple $\kk$-linear category} means an additive,
Karoubian, Hom-finite Schur category in which every object is a finite direct
sum of indecomposable objects; these indecomposable objects are then called
simple.
\end{definition}

In this setting only finitely many isomorphism classes of underlying objects
can occur as subobjects of a fixed object $X$.  Indeed, decompose $X$ into
pairwise nonisomorphic simple objects with finite multiplicities.  A
monomorphism into $X$ induces an injection on each space
$\Hom_{\C}(S,-)$ with $S$ simple, so its source contains no additional
simple object and no larger multiplicity.  This assertion concerns the
isomorphism class of the source, not the equivalence class of its embedding
over $X$.

\begin{lemma}[Finitely many split source objects]\label{lem:split-source-types}
Over any field $\kk$, a fixed object $X$ in a Hom-finite $\kk$-linear
category admits only finitely many isomorphism classes of sources of split
monomorphisms into $X$.
\end{lemma}

\begin{proof}
The additive Karoubi completion of $\C$ is Hom-finite and hence
Krull--Schmidt \cite[Theorem~6.1]{Shah}.  Write
$X\cong\bigoplus_j Q_j^{\oplus n_j}$ there, with pairwise nonisomorphic
indecomposable $Q_j$.  Every split source is a direct summand, hence has
the form $\bigoplus_j Q_j^{\oplus m_j}$ with $0\leq m_j\leq n_j$.
There are finitely many such tuples, and the fully faithful inclusion of
$\C$ into its completion reflects isomorphisms.
\end{proof}

This concerns isomorphism classes of sources, not of embeddings over $X$.

\section{Low-degree Hochschild contraction}\label{sec:hoch}

In this section, $\C$ may be any $\kk$-linear monoidal category over any
field $\kk$.  Neither Hom-finiteness nor algebraic closedness is required;
no additive, abelian, or rigid structure is used.
Let $A$ be an algebra object and $M$ an $A$--$A$-bimodule.  Put
$C^n(A,M):=\Hom_{\C}(A^{\otimes n},M)$ for $n\geq0$, where
$A^{\otimes0}=\one$.  We write
$d^n:C^n(A,M)\to C^{n+1}(A,M)$ for the Hochschild differential in
degree $n$.
These cochain operators are distinct from the geometric differential
$\mathrm d$ of a regular map used later.
They are the usual Hochschild operators, expressed through the
multiplication and bimodule actions; see \cite[\S9.1]{Weibel} and
\cite{AMS}.  The three operators needed here are
\begin{align}
 d^0z
 &=\ell_M\circ(\id_A\otimes z)-r_M\circ(z\otimes\id_A),
 \label{eq:d0}\\
 d^1f
 &=\ell_M\circ(\id_A\otimes f)-f\circ m_A
   +r_M\circ(f\otimes\id_A),
 \label{eq:d1}\\
 d^2\varphi
 &=\ell_M\circ(\id_A\otimes\varphi)
   -\varphi\circ(m_A\otimes\id_A)
   +\varphi\circ(\id_A\otimes m_A)
   -r_M\circ(\varphi\otimes\id_A).
 \label{eq:d2}
\end{align}
The algebra and bimodule axioms give
$d^1\circ d^0=0$ and $d^2\circ d^1=0$.  For $n=1,2$, write
$Z^n(A,M):=\ker d^n$, $B^n(A,M):=\operatorname{im}d^{n-1}$, and
$\HH^n(A,M):=Z^n(A,M)/B^n(A,M)$.  Thus $Z^1(A,M)$ consists of derivations
and $B^1(A,M)$ of inner derivations: in vector spaces, these mean
$f(ab)=af(b)+f(a)b$ and $f(a)=az-za$ for some $z\in M$, respectively.
Elements of $Z^2(A,M)$ are $2$-cocycles and elements of $B^2(A,M)$ are
$2$-coboundaries.  These kernels, images, and quotients are taken in
ordinary $\kk$-vector spaces, so their definition requires neither
finite-dimensional morphism spaces nor kernels or cokernels in $\C$.

The contraction below is the categorical form of the standard separability
homotopy; compare \cite{AMS}.  Ostrik already outlines the
Hochschild-cohomological approach to algebra-object rigidity in
\cite{Ostrik}.  Here the required low-degree contraction
is an explicit identity of morphisms in a category that need not be abelian.

\begin{proposition}[Low-degree separability vanishing]
\label{prop:hochschild-vanishing}
If $A$ is separable, then for every $A$--$A$-bimodule $M$,
$\HH^1(A,M)=\HH^2(A,M)=0$.
\end{proposition}

\begin{proof}
Choose a separability morphism $p_A:\one\to A\otimes A$ as in
Lemma~\ref{lem:separability-morphism}.  For $n\geq1$ define
\begin{equation}\label{eq:categorical-contraction}
 H_n:C^n(A,M)\longrightarrow C^{n-1}(A,M),\qquad
 H_n(\psi):=\ell_M\circ(\id_A\otimes\psi)
 \circ(p_A\otimes\id_A^{\otimes(n-1)}).
\end{equation}
In particular,
\begin{align}
 H_1(f)&=\ell_M\circ(\id_A\otimes f)\circ p_A,
 \label{eq:h1-contraction}\\
 H_2(\varphi)&=\ell_M\circ(\id_A\otimes\varphi)
 \circ(p_A\otimes\id_A).
 \label{eq:h2-contraction}
\end{align}
To verify the contraction identities, put
\[
 L_p:=(m_A\otimes\id_A)\circ(\id_A\otimes p_A),\qquad
 R_p:=(\id_A\otimes m_A)\circ(p_A\otimes\id_A).
\]
By Lemma~\ref{lem:separability-morphism}, $L_p=R_p=s_A$.

\medskip
\noindent\textbf{Contraction lemma.}
For $f\in C^1(A,M)$ and $\varphi\in C^2(A,M)$, one has
\begin{align}
 d^0(H_1(f))
 &=\ell_M\circ(\id_A\otimes f)\circ L_p
   -r_M\circ(H_1(f)\otimes\id_A),\label{eq:h1-first}\\
 H_2(d^1f)
 &=f-\ell_M\circ(\id_A\otimes f)\circ R_p
   +r_M\circ(H_1(f)\otimes\id_A),\label{eq:h1-second}\\
 d^1(H_2(\varphi))
 &=\ell_M\circ(\id_A\otimes\varphi)
     \circ(L_p\otimes\id_A)
   -H_2(\varphi)\circ m_A
   +r_M\circ(H_2(\varphi)\otimes\id_A),\label{eq:h2-first}\\
 H_3(d^2\varphi)
 &=\varphi-\ell_M\circ(\id_A\otimes\varphi)
     \circ(R_p\otimes\id_A)
   +H_2(\varphi)\circ m_A
   -r_M\circ(H_2(\varphi)\otimes\id_A).
 \label{eq:h2-second}
\end{align}
For the first and third identities, substitute the definitions of $d^0$
and $d^1$ and use left-module associativity to expose $L_p$ in the
left-action term.  In the second and fourth identities, expand
$H_2(d^1f)$ and $H_3(d^2\varphi)$, respectively.  Left-module
associativity, $m_A\circ p_A=e_A$, and left unitality turn the first
summand into $f$ or $\varphi$.  The middle summands follow by
functoriality of the tensor product and the definition of $R_p$;
compatibility of the two actions moves the final right action outside
$H_n$.  Thus these four expansions use normalization and the bimodule
axioms.  Centrality, namely $L_p=R_p$, identifies the terms that cancel
when the equations are added.

Adding \eqref{eq:h1-first} to \eqref{eq:h1-second}, and
\eqref{eq:h2-first} to \eqref{eq:h2-second}, gives
\begin{align}
 d^0\circ H_1+H_2\circ d^1&=\id_{C^1(A,M)},
 \label{eq:homotopy-one}\\
 d^1\circ H_2+H_3\circ d^2&=\id_{C^2(A,M)}.
 \label{eq:homotopy-two}
\end{align}
Consequently $d^1f=0$ implies $f=d^0(H_1(f))$, while
$d^2\varphi=0$ implies $\varphi=d^1(H_2(\varphi))$.  Hence both
cohomology groups vanish.
\end{proof}

\section{Affine geometry and degree}\label{sec:geometry}

This section recalls the elementary algebraic geometry used below; see
\cite[Chapters~1--4 and 9]{CoxLittleOShea} for the affine notions and
\cite[Proposition~1.8]{Borel} for algebraic-group orbits in arbitrary
characteristic.  An affine variety over
$\kk$ means here a reduced common zero locus in some affine space
$\mathbb A^n_{\kk}$; it need not be irreducible.  Its Zariski-closed subsets
are the common zero loci of polynomial families.  A nonempty space is
\emph{irreducible} if it is not the union of two proper closed subsets, and
the irreducible components of a variety are its maximal irreducible closed
subsets; there are finitely many of them.  We denote their set by
$\Irr(V)$.

A \emph{quasi-affine variety} is an open subvariety of an affine variety,
or equivalently a locally closed subvariety of an affine space.  The
\emph{dimension} of a nonempty (quasi-)affine variety $V$ is the largest
integer $r$ for which there is a chain
\[
 Z_0\subsetneq Z_1\subsetneq\cdots\subsetneq Z_r\subseteq V
\]
of nonempty irreducible closed subsets of $V$.  Its \emph{local dimension}
at $v\in V$ is
\[
 \dim_vV:=\min_{U\ni v}\dim U
   =\max\{\dim C:C\in\Irr(V),\ v\in C\},
\]
where $U$ ranges over open neighborhoods of $v$ in $V$; equivalently,
$\dim_vV$ is the Krull dimension of the local ring $\mathcal O_{V,v}$.
For affine $V$, $\dim V$ is the Krull dimension of $\kk[V]$; see
\cite[AG, \S\S{}1.4, 3.8, and 9.1]{Borel}.

For affine $V$, if the radical ideal $I(V)\subseteq\kk[x_1,\ldots,x_n]$ is generated by
$q_1,\ldots,q_s$ and $v\in V$, its Zariski tangent space is
\[
 T_vV=\{w\in\kk^n:Dq_j(v)w=0\text{ for }1\leq j\leq s\},
 \qquad
 Dq(v)w:=\sum_{i=1}^n\frac{\partial q}{\partial x_i}(v)w_i.
\]
Here $Dq(v)$ is the differential of the polynomial $q$ at $v$.
For a regular map $F:V\to W$, we write
$(\mathrm dF)_v:T_vV\to T_{F(v)}W$ for its induced linear map on tangent
spaces; in affine coordinates this is the restriction of the Jacobian
map $DF(v)$.  Thus $D$ and $\mathrm d$ denote the same differential
construction in coordinates and intrinsically, respectively
\cite[AG, \S{}16.1]{Borel}.
Let $R=\kk[\varepsilon]/(\varepsilon^2)$, where $\varepsilon$ is a formal
symbol with $\varepsilon^2=0$.  Every element of $R$ is uniquely
$a+\varepsilon b$, and
\[
 q(v+\varepsilon w)=q(v)+\varepsilon Dq(v)w.
\]
Thus a tangent vector is a first-order lift $v+\varepsilon w$ satisfying
all equations in $I(V)$ over $R$.  Terms containing two or more copies of
$\varepsilon$ vanish; this is the linearization used below.
The definition is independent of the chosen generators of $I(V)$.
For a quasi-affine variety, the tangent space is computed in any affine
open neighborhood of the point.
Linearizing equations defining $V$ only as a zero set always gives
necessary conditions on tangent vectors, but these need not be sufficient.
The local dimension satisfies $\dim_vV\leq\dim_{\kk}T_vV$, with equality
precisely when $v$ is smooth.

Algebraic groups are understood here as reduced algebraic varieties with
regular multiplication and inversion; over $\kk$ they are smooth
\cite[Proposition~1.2]{Borel}.  All groups used below are affine.
An orbit in an affine variety is smooth and locally closed
\cite[Proposition~1.8]{Borel}; with its reduced induced structure it is
therefore a quasi-affine variety.

\begin{lemma}[Surjective orbit differential]\label{lem:open-orbit}
Let an algebraic group $G$ over $\kk$ act algebraically on an affine variety
$V$, let $v\in V$, and let
$\theta_v:G\to V$, $g\mapsto g\boldsymbol\cdot v$, be the orbit map.  If
$(\mathrm d\theta_v)_{1_G}:T_{1_G}G\to T_vV$ is surjective, then
$G\boldsymbol\cdot v$ is open in $V$.  Consequently $V$ has only finitely
many orbits satisfying this differential-surjectivity condition.
\end{lemma}

\begin{proof}
Let $G^\circ$ be the identity component, put
$O=G^\circ\boldsymbol\cdot v$, and let $Z=\overline O$.
The subgroup $G^\circ$ is open and irreducible; see
\cite[Proposition~1.2, pp.~46--47]{Borel}.
Restricting the orbit map to $G^\circ$ therefore leaves its differential
at $1_G$ unchanged.  Since the restricted map factors through $O$,
\[
 \operatorname{im}(\mathrm d\theta_v)_{1_G}
 \subseteq T_vO\subseteq T_vV.
\]
Surjectivity gives $T_vO=T_vV$.  The orbit $O$ is smooth and open in the
irreducible variety $Z$ by \cite[Proposition~1.8, p.~53]{Borel}.  Hence
\[
 \dim Z=\dim O=\dim_{\kk}T_vO=\dim_{\kk}T_vV.
\]
Because $Z\subseteq V$ passes through $v$, we also have
\[
 \dim Z\leq\dim_vV\leq\dim_{\kk}T_vV,
\]
where the second inequality is
\cite[AG, \S{}3.9 and Theorem~17.1]{Borel}.
Equality follows throughout.  Thus $v$ is a smooth point of $V$ and lies
on a unique irreducible component $C$; see
\cite[AG, Theorem~17.1, p.~40]{Borel}.
Since $Z\subseteq C$ and $\dim Z=\dim C$, we have $Z=C$.
Therefore $O$ is open in $C$.

Removing from $V$ the union of the irreducible components other than $C$
gives an open neighborhood of $v$ contained in $C$.  Intersecting it with
$O$ shows that $O$ contains an open neighborhood of $v$ in $V$; translating
this neighborhood by $G^\circ$ proves that $O$ is open in $V$.  Hence
$G\boldsymbol\cdot v=\bigcup_{g\in G}gO$ is open as well.

Finally, $V$ has finitely many irreducible components, and an irreducible
component meets at most one open orbit because any two nonempty open
subsets of an irreducible variety intersect.  Thus only finitely many
open orbits occur.
\end{proof}

\begin{lemma}[Zero tangent spaces imply finiteness]\label{lem:zero-tangent}
If $V$ is an affine variety and $T_vV=0$ for every $v\in V$, then $V$ is a
finite set.
\end{lemma}

\begin{proof}
The inequality $\dim_vV\leq\dim_{\kk}T_vV$ makes $V$ zero-dimensional.  Its
coordinate ring is therefore an Artinian reduced finitely generated
$\kk$-algebra, hence a finite product of copies of $\kk$.  Thus $V$ has
finitely many points.
\end{proof}

\subsection*{B\'ezout inequality}

For an irreducible projective variety $Y\subseteq\mathbb P^n_{\kk}$ of
dimension $r$, its \emph{algebraic degree} $\deg(Y)$ is the number of points
in $Y\cap L$ for a general linear subspace $L$ of codimension $r$.
Here \emph{general} means belonging to a suitable nonempty Zariski-open
subset of the Grassmannian: the intersection is finite and transverse
there, and its cardinality is constant.  Equivalently, $\deg(Y)$ is $r!$
times the leading coefficient of its Hilbert polynomial
\cite[Lecture~18]{Harris}.  The degree of an irreducible affine variety
is the degree of its projective closure.

For a reduced affine algebraic set $V\subseteq\mathbb A^n_{\kk}$, define its
\emph{cumulative degree} by
\begin{equation}\label{eq:cumulative-degree}
 \cdeg(V):=\sum_{Z\in\Irr(V)}\deg(\overline Z),
\end{equation}
where $\overline Z\subseteq\mathbb P^n_{\kk}$ is the projective closure of
$Z$.  This version of degree records components of every dimension.  In
particular,
\begin{equation}\label{eq:components-degree}
 |\Irr(V)|\leq\cdeg(V),
\end{equation}
since every irreducible component has positive integral degree.

The classical projective B\'ezout equality says that, when
pure-dimensional projective varieties $Y,H\subseteq\mathbb P^n$ satisfy
$\dim Y+\dim H\geq n$ and meet properly,
\[
 \deg(Y)\deg(H)=\sum_Z i(Z;Y,H)\deg(Z),
\]
where $Z$ runs over the irreducible components of $Y\cap H$ and
$i(Z;Y,H)$ is the positive intersection multiplicity
\cite[Theorem~18.4]{Harris}.  In particular, if $H$ is a hypersurface of
degree $e$ not containing the irreducible variety $Y$ with $\dim Y\geq1$,
this sum equals
$e\deg(Y)$.  Discarding multiplicities and components at infinity gives
the affine inequality used below.

We use the following standard affine form of B\'ezout's theorem
(cf.~\cite{Heintz}).

\begin{theorem}[Affine B\'ezout inequality]\label{thm:affine-bezout}
Let $V\subseteq\mathbb A^n_{\kk}$ be the reduced common zero locus of $s$
polynomials of degree at most $d$, where $d\geq1$ is an integer bound on
their total degrees.  Then
\begin{equation}\label{eq:affine-bezout}
 \cdeg(V)\leq d^{\min\{n,s\}}.
\end{equation}
Consequently, $V$ has at most $d^{\min\{n,s\}}$ irreducible components.
\end{theorem}

\begin{proof}
Let $W$ be a reduced affine algebraic set and $g$ a polynomial of degree
at most $d$.  Apply the projective B\'ezout theorem in the hypersurface
case to the closure of each irreducible component of $W$ not contained
in $Z(g)$.  A component contained in $Z(g)$ retains its original degree.
Since $d\geq1$, summing these estimates gives
\begin{equation}\label{eq:bezout-step}
 \cdeg\bigl(W\cap Z(g)\bigr)\leq d\,\cdeg(W).
\end{equation}
Iteration over the $s$ defining equations yields $\cdeg(V)\leq d^s$.

For the bound in terms of $n$, we prove by induction on $r=\dim W$ that
\[
 \cdeg(W\cap V)\leq\deg(W)d^r
\]
for every irreducible affine variety $W$.  The claim is immediate when
$r=0$ or $W\subseteq V$.  Otherwise, some defining polynomial $g$ of $V$
does not vanish identically on $W$.  Every irreducible component $Y$ of
$W\cap Z(g)$ has dimension $r-1$, and the same B\'ezout inequality gives
$\sum_Y\deg(Y)\leq d\deg(W)$.  Since
$W\cap V=\bigcup_Y(Y\cap V)$, induction gives
\[
 \cdeg(W\cap V)
 \leq\sum_Y\cdeg(Y\cap V)
 \leq d^{r-1}\sum_Y\deg(Y)
 \leq d^r\deg(W).
\]
Apply this with $W=\mathbb A^n$, of degree one, and combine the resulting
$d^n$ bound with the $d^s$ bound.
\end{proof}

We record the following result, which will be used later.

\Needspace{7\baselineskip}
\begin{lemma}[Counting open orbits and isolated points]\label{lem:degree-counting}
Let $V$ be a reduced affine algebraic set.
\begin{enumerate}[label=\textup{(\roman*)}]
 \item If an algebraic group acts on $V$, the number of open orbits is at
 most $\cdeg(V)$.
 \item The number of isolated points of $V$ is at most $\cdeg(V)$.
\end{enumerate}
\end{lemma}

\begin{proof}
An irreducible component meets at most one open orbit: two nonempty open
subsets of an irreducible space meet, whereas distinct orbits do not.  This
proves \textup{(i)} by \eqref{eq:components-degree}.  An isolated point is a
zero-dimensional irreducible component, of degree one, which proves
\textup{(ii)}.
\end{proof}

\section{Algebra structures on a fixed object}\label{sec:structures}

Throughout this section, $\C$ is a Hom-finite $\kk$-linear monoidal
category, where $\kk$ is algebraically closed, and $B$ is a fixed object in
$\C$.  Put $\mathsf M_B:=\Hom_{\C}(B\otimes B,B)$ and
$\mathsf U_B:=\Hom_{\C}(\one,B)$.  Both are finite-dimensional
$\kk$-vector spaces.  For variables $m\in\mathsf M_B$ and
$e\in\mathsf U_B$, let $\Alg(B)$ be the reduced common zero locus in
$\mathsf M_B\oplus\mathsf U_B$ of the polynomial maps encoded by
\begin{align}
  m\circ(m\otimes\id_B)&=m\circ(\id_B\otimes m),\label{eq:alg-assoc}\\
  m\circ(e\otimes\id_B)&=\id_B
  =m\circ(\id_B\otimes e).\label{eq:alg-unit}
\end{align}
Thus $\Alg(B)$ is an affine variety whose points are precisely the unital
algebra structures on the fixed object $B$.  This construction is independent
of bases: changing bases in the source spaces gives an invertible linear
coordinate change, while changing bases in the target spaces merely
recombines the equations by invertible linear combinations.  Here
``reduced'' means that the ideal generated by the coordinate equations is
replaced by its radical; this does not change the common zero set over the
algebraically closed field $\kk$.  The group
\[
  G_B:=\Aut_{\C}(B)=\End_{\C}(B)^\times
\]
is an algebraic group over $\kk$: it is the unit group of a
finite-dimensional $\kk$-algebra, hence a smooth Zariski-open subset of
$\End_{\C}(B)$, with $T_{\id_B}G_B=\End_{\C}(B)$.  It acts by transport of structure,
\begin{equation}\label{eq:transport-structure}
  g\boldsymbol\cdot(m,e)
  =\bigl(g\circ m\circ(g^{-1}\otimes g^{-1}),\,g\circ e\bigr).
\end{equation}
Its orbits are exactly the isomorphism classes of unital algebra structures
on the object $B$.

\subsection{The tangent calculation with a varying unit}

Fix $(m,e)\in\Alg(B)$ and write $A=(B,m,e)$.  Every Zariski tangent
vector is represented by a pair
$(\varphi,v)\in\mathsf M_B\oplus\mathsf U_B$.  Because $\Alg(B)$ is the
\emph{reduced} zero locus of the structure equations, such a tangent vector
necessarily satisfies their linearizations; see the tangent-space
description in \S{}\ref{sec:geometry} and \cite[AG, \S\S{}16.1--16.2]{Borel}.
We do not need the converse.
Linearized associativity is
\begin{equation}\label{eq:linearized-assoc}
  \varphi\circ(m\otimes\id_B)+m\circ(\varphi\otimes\id_B)
  =
  \varphi\circ(\id_B\otimes m)+m\circ(\id_B\otimes\varphi),
\end{equation}
which is precisely $d^2\varphi=0$ for the regular $A$-bimodule.  The two
linearized unit identities are
\begin{align}
  \varphi\circ(e\otimes\id_B)+m\circ(v\otimes\id_B)&=0,
  \label{eq:linearized-left-unit}\\
  \varphi\circ(\id_B\otimes e)+m\circ(\id_B\otimes v)&=0.
  \label{eq:linearized-right-unit}
\end{align}

Assume now that $A$ is separable.  By
Proposition~\ref{prop:hochschild-vanishing}, there is a morphism $h:B\to B$
with $\varphi=d^1h$.
Precomposing \eqref{eq:linearized-left-unit} with
$e:\one\to B$ gives
\begin{equation}\label{eq:h-on-unit}
  \varphi\circ(e\otimes e)+m\circ(v\otimes e)=0.
\end{equation}
The right-unit identity gives $m\circ(v\otimes e)=v$.  Moreover,
\eqref{eq:d1} for the regular bimodule, precomposed with
$e\otimes e$, gives
\[
  (d^1h)\circ(e\otimes e)=h\circ e.
\]
Hence
\begin{equation}\label{eq:h-unit-v}
  h\circ e=-v.
\end{equation}

Let
\[
  \theta_{(m,e)}:G_B\longrightarrow\mathsf M_B\oplus\mathsf U_B,
  \qquad
  g\longmapsto g\boldsymbol\cdot(m,e)
\]
be the orbit map.  Its ordinary differential at $\id_B$ is
\begin{equation}\label{eq:orbit-differential-alg}
  (\mathrm d\theta_{(m,e)})_{\id_B}(f)
  =\bigl(f\circ m-m\circ(f\otimes\id_B)-m\circ(\id_B\otimes f),\,
  f\circ e\bigr)
  =\bigl(-d^1f,\,f\circ e\bigr).
\end{equation}
This follows by differentiating multiplication and inversion in the
finite-dimensional algebra $\End_{\C}(B)$; it does not require a scalar
extension of the object $B$.  Taking $f=-h$, equations
\eqref{eq:h-unit-v} and $\varphi=d^1h$ give
\[
  (\mathrm d\theta_{(m,e)})_{\id_B}(-h)=(\varphi,v).
\]
Thus every tangent vector to $\Alg(B)$ lies in the image of the orbit-map
differential.  The reverse inclusion is automatic because the orbit map
takes values in $\Alg(B)$.  We have proved
\begin{equation}\label{eq:full-tangent-alg}
  \operatorname{im}(\mathrm d\theta_{(m,e)})_{\id_B}
  =T_{(m,e)}\Alg(B)
\end{equation}
at every separable point.

\begin{theorem}[Fixed-object rigidity]\label{thm:fixed-object}
Let $B$ be any object of a Hom-finite $\kk$-linear monoidal category.  The
object $B$ supports only finitely many isomorphism classes of separable
unital algebra structures.
\end{theorem}

\begin{proof}
Equation~\eqref{eq:full-tangent-alg} and
Lemma~\ref{lem:open-orbit} show that the $G_B$-orbit of every separable point
of $\Alg(B)$ is open.  The same lemma shows that $\Alg(B)$ has only finitely
many open $G_B$-orbits.
\end{proof}

In the characteristic-zero fusion setting, Ostrik's reconstruction theorem
\cite[Theorem~3.1]{Ostrik} relates algebra objects to module categories;
compare rigidity and finiteness of semisimple module categories with a
prescribed number of simple isomorphism classes in \cite[Corollary~2.35 and \S{}7]{ENO}.  Theorem~\ref{thm:fixed-object}
fixes the underlying object and counts algebra structures up to algebra
isomorphism, rather than module categories up to equivalence.

\subsection{Effective bounds}

For later use, set
\begin{equation}\label{eq:fixed-object-data}
\begin{aligned}
 a_B&:=\dim_{\kk}\Hom_{\C}(B\otimes B,B),
 &u_B&:=\dim_{\kk}\Hom_{\C}(\one,B),\\
 t_B&:=\dim_{\kk}\Hom_{\C}(B^{\otimes3},B),
 &E_B&:=\dim_{\kk}\End_{\C}(B).
\end{aligned}
\end{equation}

\Needspace{7\baselineskip}
\begin{theorem}[Fixed-object bound]\label{thm:fixed-object-bound}
The number of isomorphism classes of separable unital algebra structures on
$B$ is at most
\begin{equation}\label{eq:fixed-object-bound}
 2^{\min\{a_B+u_B,\,t_B+2E_B\}}.
\end{equation}
\end{theorem}

\begin{proof}
The variety $\Alg(B)$ lies in an affine space of dimension $a_B+u_B$.
Associativity contributes at most $t_B$ scalar quadratic equations, while
the two unit identities contribute at most $2E_B$ scalar equations of
degree at most two.  Hence Theorem~\ref{thm:affine-bezout} gives
\[
 \cdeg\bigl(\Alg(B)\bigr)
 \leq 2^{\min\{a_B+u_B,t_B+2E_B\}}.
\]
For every separable point, equation~\eqref{eq:full-tangent-alg} says that the
orbit-map differential is surjective; Lemma~\ref{lem:open-orbit} therefore
makes that orbit open.  Finally,
Lemma~\ref{lem:degree-counting}\textup{(i)} proves the claim.
\end{proof}

Fix now a nonzero morphism $e:\one\to B$.  Let $\Alg_e(B)$ be the
fixed-unit slice and let
\[
 G_{B,e}:=\{g\in\Aut_{\C}(B):g\circ e=e\}
\]
be its stabilizer.

The group $G_{B,e}$ is the intersection of the open set $G_B$ with the
affine linear space
\[
 \id_B+\{h\in\End_{\C}(B):h\circ e=0\}.
\]
It is therefore smooth, with
$T_{\id_B}G_{B,e}=\{h\in\End_{\C}(B):h\circ e=0\}$.

\begin{proposition}[Fixed-unit rigidity]\label{prop:fixed-unit-rigidity}
Every separable point of $\Alg_e(B)$ has an open $G_{B,e}$-orbit.
\end{proposition}

\begin{proof}
A tangent vector in the fixed-unit slice has the form $(\varphi,0)$.  The
calculation preceding \eqref{eq:h-unit-v} gives $\varphi=d^1h$ and
$h\circ e=0$.  Thus $-h$ belongs to the Lie algebra of $G_{B,e}$, and
\eqref{eq:orbit-differential-alg} shows that $(\varphi,0)$ lies in the image
of the stabilizer orbit-map differential.  That differential is therefore
surjective, so Lemma~\ref{lem:open-orbit} applies.
\end{proof}

Consider the linear map
\begin{equation}\label{eq:lambda-fixed-unit}
 \Lambda_{B,e}:\Hom_{\C}(B\otimes B,B)\longrightarrow\End_{\C}(B)^2,
 \qquad
 m\longmapsto\bigl(m\circ(e\otimes\id_B),
 m\circ(\id_B\otimes e)\bigr),
\end{equation}
and put
\begin{equation}\label{eq:q-fixed-unit}
 q_{B,e}:=\dim_{\kk}\ker\Lambda_{B,e}.
\end{equation}
When nonempty, the unit equations cut out an affine space of this dimension.

\begin{theorem}[Fixed-unit bound]\label{thm:fixed-unit-bound}
The number of separable algebra structures on $B$ having unit exactly
$e$, modulo algebra isomorphisms fixing $e$, is at most
\begin{equation}\label{eq:fixed-unit-bound}
 2^{\min\{q_{B,e},t_B\}}.
\end{equation}
If $e$ has a retraction $\pi:B\to\one$, then
\begin{equation}\label{eq:q-exact}
 q_{B,e}=a_B-2E_B+u_B.
\end{equation}
\end{theorem}

\begin{proof}
On the affine unit slice, associativity is a system of at most $t_B$
quadratic equations.  Proposition~\ref{prop:fixed-unit-rigidity},
Theorem~\ref{thm:affine-bezout}, and
Lemma~\ref{lem:degree-counting}\textup{(i)} give
\eqref{eq:fixed-unit-bound}.

Assume $\pi\circ e=\id_{\one}$.  The image of
$\Lambda_{B,e}$ is exactly
\begin{equation}\label{eq:image-lambda}
 \{(L,R)\in\End_{\C}(B)^2:L\circ e=R\circ e\}.
\end{equation}
The inclusion from left to right follows by evaluation on
$e\otimes e$.  Conversely, if $L\circ e=R\circ e=:v$, then
\[
 m=L\circ(\pi\otimes\id_B)
   +R\circ(\id_B\otimes\pi)
   -v\circ(\pi\otimes\pi)
\]
has $\Lambda_{B,e}(m)=(L,R)$.  The map
$(L,R)\mapsto L\circ e-R\circ e$ from $\End_{\C}(B)^2$ onto
$\Hom_{\C}(\one,B)$ is surjective, since $w$ is the image of
$(w\circ\pi,0)$.  Thus the space in \eqref{eq:image-lambda} has
dimension $2E_B-u_B$, and rank--nullity gives \eqref{eq:q-exact}.
\end{proof}

\begin{corollary}[Connected fixed object]\label{cor:connected-fixed-object}
Suppose $u_B=1$ and a nonzero morphism $e:\one\to B$ has a retraction;
this holds, for example, when $\C$ is semisimple with simple tensor unit
and $B$ is connected.
Then the total number of separable algebra-isomorphism classes on $B$ is at
most
\begin{equation}\label{eq:connected-fixed-object}
 2^{\min\{a_B-2E_B+1,t_B\}}.
\end{equation}
\end{corollary}

\begin{proof}
All nonzero maps $\one\to B$ are scalar multiples of one another, and a
scalar automorphism of $B$ transports any unit to the chosen one.  Apply
Theorem~\ref{thm:fixed-unit-bound}.
\end{proof}
\section{Homomorphisms and inner conjugacy}\label{sec:homomorphisms}

Deformations of algebra morphisms and diagrams were developed by
Gerstenhaber--Schack \cite{GerstenhaberSchack83,GerstenhaberSchack85};
see also \cite[\S\S{}5--6]{FregierMarklYau} for a later treatment of the
deformation complex.  Here the target algebra is fixed, and separability
identifies infinitesimal variations with inner-conjugation directions.

Throughout this section, $\C$ is a Hom-finite $\kk$-linear monoidal
category over an algebraically closed field $\kk$.  Fix a separable algebra
$A=(A,m_A,e_A)$ and an arbitrary algebra $X=(X,m_X,e_X)$.  Inside the
finite-dimensional affine space
$\Hom_{\C}(A,X)$, let $\Hom_{\mathrm{alg}}(A,X)$ be the reduced common zero
locus of the polynomial maps
\[
  f\longmapsto f\circ e_A-e_X,
  \qquad
  f\longmapsto f\circ m_A-m_X\circ(f\otimes f).
\]
Thus its points are precisely the unital algebra morphisms $A\to X$.

At $f\in\Hom_{\mathrm{alg}}(A,X)$, every Zariski tangent vector $\xi$
necessarily satisfies the linearized equations
\begin{align}
  \xi\circ e_A&=0,\label{eq:tangent-hom-unit}\\
  \xi\circ m_A
  &=m_X\circ(\xi\otimes f)+m_X\circ(f\otimes\xi).
  \label{eq:tangent-hom-mult}
\end{align}
Use $f$ to define actions
\begin{equation}\label{eq:f-actions}
  \ell_f:=m_X\circ(f\otimes\id_X):A\otimes X\longrightarrow X,
  \qquad
  r_f:=m_X\circ(\id_X\otimes f):X\otimes A\longrightarrow X.
\end{equation}
Since $f$ is a unital algebra homomorphism, associativity and unitality of
$X$ show that these morphisms make $X$ an $A$--$A$-bimodule, denoted
${}_fX_f$.

With these actions, equation~\eqref{eq:tangent-hom-mult} is precisely the
derivation identity.  The unit equation~\eqref{eq:tangent-hom-unit} is part
of the linearized definition of the homomorphism variety.  Conversely, every
derivation satisfies it automatically: evaluating the derivation identity on
$e_A\otimes e_A$ gives
$\xi\circ e_A=2\xi\circ e_A$, hence $\xi\circ e_A=0$.  Therefore
\begin{equation}\label{eq:tangent-hom-z1}
  T_f\Hom_{\mathrm{alg}}(A,X)\subseteq Z^1(A,{}_fX_f).
\end{equation}

The unit group $\Gamma(X)^\times$ is a smooth affine algebraic group over
$\kk$, with $T_{e_X}\Gamma(X)^\times=\Gamma(X)$.
Indeed, for a polynomial $q$ on a vector space $V$, write
$D(q):=\{v\in V:q(v)\ne0\}$.  Left multiplication in the
finite-dimensional $\kk$-algebra $\Gamma(X)$ gives a linear map
\[
  L:\Gamma(X)\longrightarrow\End_{\kk}(\Gamma(X)),
  \qquad u\longmapsto L_u.
\]
Put $q:=\det\circ L$.  If $u$ is invertible, then $q(u)\ne0$.  Conversely,
if $q(u)\ne0$, set $v=L_u^{-1}(e_X)$, so $u\star v=e_X$.  Associativity gives
$L_u(v\star u-e_X)=0$, and the injectivity of $L_u$ gives
$v\star u=e_X$.  Thus $\Gamma(X)^\times=D(q)$.  Multiplication is
polynomial, while
\[
 u^{-1}=L_u^{-1}(e_X)
\]
is regular on $D(q)$, which is open in the vector space $\Gamma(X)$.
Hence $\Gamma(X)^\times$ is a smooth affine algebraic
group, and its conjugation action \eqref{eq:conjugation-action} is algebraic.

For $f\in\Hom_{\mathrm{alg}}(A,X)$, let
\[
  \omega_f:\Gamma(X)^\times\longrightarrow
  \Hom_{\mathrm{alg}}(A,X),
  \qquad u\longmapsto u\boldsymbol\cdot f
\]
be the orbit map.  Its Zariski differential at the identity $e_X$ is
\begin{equation}\label{eq:orbit-differential-hom}
  (\mathrm d\omega_f)_{e_X}(z)
  =m_X\circ(z\otimes f)-m_X\circ(f\otimes z):A\longrightarrow X.
\end{equation}
Formula~\eqref{eq:d0} for the actions \eqref{eq:f-actions} shows that
\eqref{eq:orbit-differential-hom} is $-d^0z$.  Thus the image of the
orbit-map differential is $B^1(A,{}_fX_f)$, and it is contained in the
tangent space of the homomorphism variety.  Separability and
Proposition~\ref{prop:hochschild-vanishing} give the
following chain of inclusions and equalities:
\begin{equation}\label{eq:full-tangent-hom}
  \operatorname{im}(\mathrm d\omega_f)_{e_X}
  =B^1(A,{}_fX_f)
  \subseteq T_f\Hom_{\mathrm{alg}}(A,X)
  \subseteq Z^1(A,{}_fX_f)
  =B^1(A,{}_fX_f).
\end{equation}

\begin{theorem}[Homomorphism rigidity]\label{thm:homomorphisms}
Let $A$ be a separable algebra object and $X$ any algebra object in a
Hom-finite $\kk$-linear monoidal category.
\begin{enumerate}[label=\textup{(\roman*)}]
  \item There are only finitely many inner-conjugacy classes of unital
  algebra homomorphisms $A\to X$.
  \item If $X$ is connected, then $\Hom_{\mathrm{alg}}(A,X)$ is a finite
  set.
\end{enumerate}
\end{theorem}

\begin{proof}
Equation~\eqref{eq:full-tangent-hom} says that every orbit-map differential
is surjective.  Lemma~\ref{lem:open-orbit} therefore shows that every
$\Gamma(X)^\times$-orbit is open
in $\Hom_{\mathrm{alg}}(A,X)$, and there are only finitely many such orbits.

If $X$ is connected, then $\Gamma(X)=\kk e_X$ and inner conjugation is
trivial.  Consequently every orbit is a singleton, so the homomorphism set
is finite.  Equivalently, every inner derivation is zero;
\eqref{eq:full-tangent-hom} then gives zero tangent space at every point, and
Lemma~\ref{lem:zero-tangent} applies.
\end{proof}

\subsection*{Effective bound}

The unit condition cuts out either the empty set or an affine translate of
\begin{equation}\label{eq:unit-kernel-hom}
 K_{A,X}:=\ker\bigl(\Hom_{\C}(A,X)\longrightarrow\Hom_{\C}(\one,X),
 f\longmapsto f\circ e_A\bigr).
\end{equation}
Set
\begin{equation}\label{eq:hom-data}
 b_{A,X}:=\dim_{\kk}K_{A,X},
 \qquad
 c_{A,X}:=\dim_{\kk}\Hom_{\C}(A\otimes A,X).
\end{equation}

\begin{theorem}[Effective homomorphism rigidity]\label{thm:hom-bound}
The number of inner-conjugacy classes of unital algebra homomorphisms
$A\to X$ is at most
\begin{equation}\label{eq:hom-bound}
 2^{\min\{b_{A,X},c_{A,X}\}}.
\end{equation}
If $X$ is connected, the same expression bounds the cardinality of
$\Hom_{\mathrm{alg}}(A,X)$.  If, in addition, this set is nonempty, then
\begin{equation}\label{eq:b-connected}
 b_{A,X}=\dim_{\kk}\Hom_{\C}(A,X)-1.
\end{equation}
\end{theorem}

\begin{proof}
If the affine unit slice is empty, the assertion is immediate.
Otherwise, inside this slice of dimension $b_{A,X}$, multiplicativity is a
system of at most $c_{A,X}$ quadratic equations.  Its inner-conjugacy orbits
are open by \eqref{eq:full-tangent-hom} and
Lemma~\ref{lem:open-orbit}.  The bound follows from
Theorem~\ref{thm:affine-bezout} and
Lemma~\ref{lem:degree-counting}\textup{(i)}.  If $X$ is connected, inner
conjugation is trivial.  Finally, when a unital map exists, the linear map in
\eqref{eq:unit-kernel-hom} has nonzero image in the one-dimensional space
$\Hom_{\C}(\one,X)$, hence rank one.  Thus rank--nullity proves
\eqref{eq:b-connected}.
\end{proof}
\section{Separable algebra subobjects}\label{sec:subobjects}

Let $X$ be an algebra object in a Hom-finite $\kk$-linear monoidal category
over an algebraically closed field $\kk$.  We now combine the two local
results for any prescribed family of subalgebras.  The only additional
input is finiteness of the underlying source objects in that family.

\begin{theorem}[Finiteness of separable algebra subobjects]
\label{thm:subobjects}
Let $X$ be an algebra object in a Hom-finite $\kk$-linear monoidal category,
and let $\mathscr S$ be a family of separable algebra subobjects of $X$.
Assume that only finitely many isomorphism classes of underlying source
objects occur in $\mathscr S$.
\begin{enumerate}[label=\textup{(\roman*)}]
  \item The members of $\mathscr S$ form finitely many inner-conjugacy
  classes.
  \item If $X$ is connected, they form finitely many equivalence classes
  over $X$.
\end{enumerate}
The algebra $X$ itself need not be separable.  In particular, the result
applies to all separable subalgebras whenever $X$ has finitely many
isomorphism classes of underlying subobjects.  Conversely, finiteness of
the inner-conjugacy classes represented in $\mathscr S$ forces finiteness
of its underlying source isomorphism classes.
\end{theorem}

\begin{proof}
Choose representatives $B_1,\ldots,B_N$ for the isomorphism classes of
underlying source objects occurring in $\mathscr S$.  By
Theorem~\ref{thm:fixed-object}, each $B_r$ supports only finitely many
separable algebra structures up to algebra isomorphism.

Let $(A,i)\in\mathscr S$ have underlying object
isomorphic to $B_r$, and choose an object isomorphism
$\phi:B_r\xrightarrow{\sim}A$.  Transport the algebra structure and the
embedding to $B_r$ by
\[
 m_{B_r}:=\phi^{-1}\circ m_A\circ(\phi\otimes\phi),\qquad
 e_{B_r}:=\phi^{-1}\circ e_A,\qquad
 i_{B_r}:=i\circ\phi.
\]
Then $\phi:(B_r,m_{B_r},e_{B_r})\to(A,m_A,e_A)$ is an algebra
isomorphism, $i_{B_r}$ is an algebra monomorphism, and separability is
preserved.  Running over all $r$ therefore leaves a finite list of
representative separable source algebras $A_1,\ldots,A_s$.

For each fixed $A_j$, Theorem~\ref{thm:homomorphisms}\textup{(i)} gives
only finitely many inner-conjugacy classes of algebra maps $A_j\to X$.
Restricting to representatives of $\mathscr S$ and taking the finite
union over $j$ proves
\textup{(i)}.  If $X$ is connected, inner conjugacy is trivial, and
\textup{(ii)} follows.  The converse holds because inner-conjugate algebra
subobjects have isomorphic underlying source objects.
\end{proof}

\begin{corollary}[Split subalgebras]\label{cor:split-subobjects}
Let $X$ be an algebra object in a Hom-finite $\kk$-linear monoidal category.
Its separable split algebra subobjects form finitely many inner-conjugacy
classes, and finitely many equivalence classes over $X$ when $X$ is connected.
\end{corollary}

\begin{proof}
Apply Theorem~\ref{thm:subobjects} and Lemma~\ref{lem:split-source-types}.
\end{proof}

\begin{corollary}[Semisimple ambient setting]\label{cor:semisimple}
Let $\C$ be a semisimple $\kk$-linear monoidal category, and let $X$ be a
connected algebra object.  Then $X$ has finitely many separable algebra
subobjects up to equivalence over $X$.  In particular, this holds in every
semisimple tensor category over an algebraically closed field.
\end{corollary}

\begin{proof}
In a semisimple category every monomorphism splits, so
Corollary~\ref{cor:split-subobjects} applies.
\end{proof}

\begin{example}[The subobject hypothesis is strictly weaker]\label{ex:strict-subobject}
Let $G=\mathbb G_a=(\kk,+)$ be the additive algebraic group, and consider
the tensor category of finite-dimensional rational $G$-representations over
$\kk$; see \cite[\S{}I.2]{Jantzen}.  This category is not semisimple.  On
$V=\kk^2$, let
\[
 \rho(t)=\begin{pmatrix}1&t\\0&1\end{pmatrix}\qquad(t\in\kk).
\]
This is a nonsplit extension of the trivial representation by itself.  Its
only invariant subspaces are $0$, the line spanned by the first basis vector,
and $V$.  Thus the hypothesis of Theorem~\ref{thm:subobjects} can hold for a
nonsemisimple object in a nonsemisimple tensor category.
\end{example}

\subsection*{Effective bounds}

For a family $\mathscr S$ as in Theorem~\ref{thm:subobjects}, choose
representatives $B_1,\ldots,B_N$ of its underlying source objects.
For each $r$, define
$a_r,u_r,t_r,E_r$ from \eqref{eq:fixed-object-data}, with $B=B_r$, and put
\begin{equation}\label{eq:subobject-data}
 h_r:=\dim_{\kk}\Hom_{\C}(B_r,X),
 \qquad
 c_r:=\dim_{\kk}\Hom_{\C}(B_r\otimes B_r,X).
\end{equation}
Discard any $B_r$ that admits no unital algebra structure with an algebra
monomorphism into $X$; such an object contributes no algebra subobject.

\begin{theorem}[Effective subobject bound]\label{thm:subobject-bound}
The number of inner-conjugacy classes represented in $\mathscr S$ is at most
\begin{equation}\label{eq:subobject-bound}
 \sum_{r=1}^N
 2^{\min\{a_r+u_r,t_r+2E_r\}+\min\{\max\{h_r-1,0\},c_r\}}.
\end{equation}
If $X$ is connected, the same expression bounds equivalence classes over
$X$.
\end{theorem}

\begin{proof}
For a fixed underlying object $B_r$,
Theorem~\ref{thm:fixed-object-bound} bounds the possible separable source
algebras.  For each such source algebra $A$,
Theorem~\ref{thm:hom-bound} gives at most
$2^{\min\{b_{A,X},c_r\}}$ inner-conjugacy classes of algebra maps into
$X$.  For every source admitting such a map, evaluation at its unit
has rank at least $\min\{1,h_r\}$: it takes the given map to $e_X$,
and the unit identity ensures $e_X\ne0$ whenever $h_r>0$.
Thus rank--nullity gives $b_{A,X}\leq\max\{h_r-1,0\}$, and the
homomorphism factor is at most
$2^{\min\{\max\{h_r-1,0\},c_r\}}$.  The monomorphisms form a subset.  Summing over
$r$ proves the result, and connectedness makes inner conjugation trivial.
\end{proof}

Suppose now that $\C$ is semisimple with simple tensor unit,
and that $X$ is connected, of finite length, and decomposes as
\begin{equation}\label{eq:target-decomposition}
 X\cong\one\oplus\bigoplus_{i=1}^t S_i^{\oplus n_i},
\end{equation}
where the $S_i\not\cong\one$ are pairwise nonisomorphic simple objects.  Set
\begin{equation}\label{eq:target-data}
\begin{aligned}
 P_X&:=\prod_{i=1}^t(n_i+1),
 &E_X&:=\dim_{\kk}\End_{\C}(X),\\
 M_X&:=\dim_{\kk}\Hom_{\C}(X\otimes X,X),
 &T_X&:=\dim_{\kk}\Hom_{\C}(X^{\otimes3},X).
\end{aligned}
\end{equation}

\begin{corollary}[Ambient-object bound]\label{cor:target-only-bound}
Let $N_{\mathrm{sep}}(X)$ denote the number of equivalence classes over $X$
of separable algebra subobjects.  Then
\begin{equation}\label{eq:target-only-bound}
 N_{\mathrm{sep}}(X)
 \leq
 P_X\,
 2^{\min\{M_X-1,T_X\}+\min\{E_X-1,M_X\}}.
\end{equation}
If $X$ is a Frobenius algebra, the same estimate applies to separable
Frobenius subalgebras in the sense of
\cite[Definition~2.6]{GhoshPalcouxExchange}.
\end{corollary}

\begin{proof}
Every algebra subobject of connected $X$ is connected.  Its underlying
object therefore has the form
\[
 \one\oplus\bigoplus_{i=1}^t S_i^{\oplus m_i},
 \qquad 0\leq m_i\leq n_i,
\]
so at most $P_X$ subobject isomorphism classes occur.  For such a source
$B$, the unit splits
and
\[
 q_{B,e}=a_B-2E_B+1\leq a_B-1\leq M_X-1,
 \qquad
 t_B\leq T_X.
\]
Moreover,
\[
 \dim_{\kk}\Hom_{\C}(B,X)-1\leq E_X-1,
 \qquad
 \dim_{\kk}\Hom_{\C}(B\otimes B,X)\leq M_X.
\]
Apply Corollary~\ref{cor:connected-fixed-object} and
Theorem~\ref{thm:hom-bound}, then sum over the $P_X$ possible isomorphism
classes.

Once an algebra embedding is fixed, the compatible Frobenius structure is
unique by Lemma~\ref{lem:frobenius-unique}.  Hence the same estimate applies.
\end{proof}

\begin{remark}[Fusion-rule form]
If $N_{ij}^{k}:=\dim_{\kk}\Hom_{\C}(S_i\otimes S_j,S_k)$ and $n_0=1$ denotes the
multiplicity of $\one=S_0$ in $X$, then
\[
 E_X=1+\sum_{i=1}^t n_i^2,
 \qquad
 M_X=\sum_{i,j,k=0}^t n_i n_j n_k N_{ij}^{k}.
\]
A similar contraction of two fusion matrices computes $T_X$.  Thus
\eqref{eq:target-only-bound} can be read directly from the fusion rules and
the multiplicity vector of $X$.
\end{remark}
\section{Frobenius subalgebras and exchange relations}\label{sec:frobenius}

Throughout this section, $\C$ is a Hom-finite $\kk$-linear monoidal
category over the algebraically closed field $\kk$.  We temporarily
assume Karoubianity for the exchange correspondence and remove it by
idempotent completion in the effective theorem.

A \emph{Frobenius algebra object} is an algebra
$(X,m_X,e_X)$ together with a coalgebra structure
\[
  \delta_X:X\longrightarrow X\otimes X,
  \qquad
  \epsilon_X:X\longrightarrow\one
\]
satisfying coassociativity, the counit identities, and the Frobenius
condition
\begin{equation}\label{eq:frobenius-equations}
  \delta_X\circ m_X
  =(m_X\otimes\id_X)\circ(\id_X\otimes\delta_X)
  =(\id_X\otimes m_X)\circ(\delta_X\otimes\id_X).
\end{equation}
The two equalities express that $\delta_X$ is a morphism of left and right
regular $X$-modules.

The Frobenius evaluation and coevaluation morphisms are
\begin{equation}\label{eq:frobenius-pairing-coeval}
  \ev_X:=\epsilon_X\circ m_X:X\otimes X\longrightarrow\one,
  \qquad
  \coev_X:=\delta_X\circ e_X:\one\longrightarrow X\otimes X.
\end{equation}
The Frobenius, unit, and counit relations give the two zig-zag identities
\begin{equation}\label{eq:snake-identities}
  (\id_X\otimes\ev_X)\circ(\coev_X\otimes\id_X)=\id_X,
  \qquad
  (\ev_X\otimes\id_X)\circ(\id_X\otimes\coev_X)=\id_X.
\end{equation}
For Frobenius pairings and induced self-duality in rigid monoidal categories,
see \cite{FuchsStigner}.

If $i:B\to X$ is a morphism between Frobenius algebras, its dual
$i^*:X\to B$ is characterized by
\begin{equation}\label{eq:dual-morphism}
  \ev_X\circ(\id_X\otimes i)
  =\ev_B\circ(i^*\otimes\id_B).
\end{equation}
Following \cite{GhoshPalcouxLattices}, a Frobenius subalgebra is
represented by a unital algebra monomorphism $i:B\to X$
satisfying
\begin{equation}\label{eq:gp-conditions}
  i^*\circ i=\id_B,
  \qquad i^{**}=i.
\end{equation}
We call it \emph{separable} when its underlying algebra object is separable.

\begin{lemma}[Uniqueness of a Frobenius subalgebra structure]
\label{lem:frobenius-unique}
In any monoidal category, let $X$ be a Frobenius algebra and fix a
unital algebra monomorphism $i:(B,m_B,e_B)\to X$.
There is at most one pair $(\delta_B,\epsilon_B)$ making $B$ a Frobenius
algebra for which $i$ is a Frobenius-subalgebra embedding.
\end{lemma}

\begin{proof}
The duality identity and $i^*\circ i=\id_B$ force
\begin{equation}\label{eq:restricted-frobenius-pairing}
 \ev_B=\ev_X\circ(i\otimes i).
\end{equation}
Thus the pairing and the counit
$\epsilon_B=\ev_B\circ(\id_B\otimes e_B)$ are fixed.
Two compatible coevaluations satisfy, by the zig-zag identities,
\begin{equation}\label{eq:coevaluation-unique}
 \coev_1=(\id_B\otimes\ev_B\otimes\id_B)
 \circ(\coev_1\otimes\coev_2)=\coev_2.
\end{equation}
Finally, the Frobenius identity recovers the comultiplication as
\begin{equation}\label{eq:comultiplication-from-coevaluation}
 \delta_B=(\id_B\otimes m_B)\circ(\coev_B\otimes\id_B).
\end{equation}
Hence the compatible Frobenius structure is unique.
\end{proof}

The identity $i^*\circ i=\id_B$ splits every Frobenius-subalgebra
embedding, whereas separability only splits multiplication.
Corollary~\ref{cor:split-subobjects} therefore gives qualitative finiteness;
the exchange argument below provides the intrinsic quantitative bound.

\begin{corollary}[Frobenius subalgebras up to inner conjugacy]
\label{cor:frobenius-inner}
Let $X$ be a Frobenius algebra object in a Hom-finite $\kk$-linear
monoidal category.  Its separable Frobenius subalgebras form finitely
many classes under inner conjugacy of their underlying algebra
subobjects, in the sense of \S2.
The ambient inner automorphism is not required to preserve the chosen
Frobenius counit.
\end{corollary}

\begin{proof}
Every such embedding splits, so Corollary~\ref{cor:split-subobjects}
gives finitely many inner-conjugacy classes of underlying algebra
subobjects.  Lemma~\ref{lem:frobenius-unique} ensures that passing to the
compatible Frobenius structures introduces no further choices over $X$.
\end{proof}

\begin{corollary}[Connected Frobenius-subalgebra finiteness]
\label{cor:frobenius}\label{cor:frobenius-tensor}
If $X$ is connected, its separable Frobenius subalgebras form finitely
many equivalence classes over $X$.  In particular, every connected
Frobenius algebra in a tensor category over an algebraically closed
field has finitely many separable Frobenius subalgebras.
\end{corollary}

\begin{proof}
Connectedness makes inner conjugacy trivial, so
Corollary~\ref{cor:frobenius-inner} and
Lemma~\ref{lem:frobenius-unique} apply.  Tensor categories in the sense of
\cite{EGNO} are Hom-finite.
\end{proof}

\subsection{The quadratic exchange locus}

For the construction of the exchange correspondence, assume first that
$\C$ is Karoubian; no semisimplicity, simple-unit, or finite-length
hypothesis is needed.  Idempotent completion will remove Karoubianity
from the numerical bound.
Let $X=(X,m_X,e_X,\delta_X,\epsilon_X)$ be a Frobenius algebra.
For $f\in\End_{\C}(X)$, write $f^*$ for its dual with respect to the
Frobenius self-duality.  The equations
\[
 b^*=b,\qquad b\circ e_X=e_X
\]
cut out the affine space
\begin{equation}\label{eq:affine-self-dual}
 \mathscr A_X
 =\id_X+\{h\in\End_{\C}(X):h^*=h,\ h\circ e_X=0\}.
\end{equation}
Write
\begin{equation}\label{eq:rX}
 r_X:=\dim_{\kk}\{h\in\End_{\C}(X):h^*=h,\ h\circ e_X=0\}.
\end{equation}
Inside $\mathscr A_X$, consider the equations
\begin{equation}\label{eq:exchange-locus}
\begin{aligned}
 b\circ b&=b,\\
 b\circ m_X\circ(b\otimes\id_X)&=m_X\circ(b\otimes b),\\
 b\circ m_X\circ(\id_X\otimes b)&=m_X\circ(b\otimes b).
\end{aligned}
\end{equation}
The last two are the \emph{exchange relations}.  We call their reduced
common zero locus, together with idempotence, the \emph{exchange locus}
$\Exch(X)$.  Its defining equations inside $\mathscr A_X$ have degree at
most two.

The description of intermediate subfactors by suitable projections goes
back to Bisch \cite{Bisch}; the exchange-relation characterization is given
in \cite{BischJones}.  Exchange-relation planar algebras
were further developed in \cite{Landau}, while the categorical
correspondence needed here is proved in \cite[Theorem~4.1]{GhoshPalcouxExchange}.

\begin{theorem}[Ghosh--Palcoux]\label{thm:exchange-correspondence}
The points of $\Exch(X)$ are naturally in bijection with equivalence
classes over $X$ of Frobenius subalgebras of $X$.  A Frobenius-subalgebra
embedding $i:B\to X$ corresponds to $b=i\circ i^*$.
\end{theorem}

\begin{proof}
Existence in both directions is proved in
\cite[Theorem~4.1]{GhoshPalcouxExchange}.  For a splitting $b=i\circ p$,
$p\circ i=\id_B$, the reconstructed multiplication and unit are
\[
 m_B=p\circ m_X\circ(i\otimes i),\qquad e_B=p\circ e_X.
\]
Two splittings of the same idempotent are canonically isomorphic over $X$;
these formulas make that isomorphism an algebra isomorphism.
Lemma~\ref{lem:frobenius-unique} makes the compatible Frobenius structure
unique, proving the asserted bijection.
\end{proof}

\subsection{Local rigidity with a varying source}

An ambient inner automorphism need not preserve the chosen Frobenius
counit.  Consequently we count the restriction of inner conjugacy of
underlying algebra subobjects, as in Corollary~\ref{cor:frobenius-inner}.
The following parameter space lets us establish local rigidity without
requiring an action on $\Exch(X)$.

The next two lemmas do not require Frobenius structure or Karoubianity.

For an object $B$, let $\mathsf P_B(X)$ be the reduced affine variety of
triples $(m,e,j)$ such that $(B,m,e)$ is an algebra object and
$j:(B,m,e)\to X$ is a unital algebra homomorphism.
The algebraic group
\[
 G_B\times\Gamma(X)^\times,
 \qquad G_B:=\Aut_{\C}(B),
\]
acts by
\begin{equation}\label{eq:joint-action}
 (g,u)\boldsymbol\cdot(m,e,j)
 =\bigl(g\circ m\circ(g^{-1}\otimes g^{-1}),\,
        g\circ e,\,c_u\circ j\circ g^{-1}\bigr).
\end{equation}

These coupled source-and-map equations are a categorical version of the
low-degree morphism-deformation equations of
\cite{GerstenhaberSchack83,GerstenhaberSchack85}, with the restricted
acting group $\Aut_{\C}(B)\times\Gamma(X)^\times$.

\begin{lemma}[Open orbits for a varying source]\label{lem:joint-open-orbit}
If $A=(B,m,e)$ is separable, the orbit of $(m,e,j)$ is open in
$\mathsf P_B(X)$.
\end{lemma}

\begin{proof}
Let $(\mu,\eta,v)\in T_{(m,e,j)}\mathsf P_B(X)$.
Projection to $\Alg(B)$ gives a tangent vector $(\mu,\eta)$.
By \eqref{eq:full-tangent-alg}, there is $z\in\End_{\C}(B)$ whose
infinitesimal change of source coordinates gives $(\mu,\eta)$.
Subtracting the tangent to \eqref{eq:joint-action} in the direction
$(z,0)$ leaves $(0,0,w)$, with $w=v+j\circ z$.
Linearizing the homomorphism identities now gives
\[
 w\circ e=0,\qquad
 w\circ m=m_X\circ(w\otimes j+j\otimes w).
\]
Thus $w$ is a derivation from $A$ to the $A$-bimodule $X$ defined by $j$.
The degree-one contraction makes it inner: for some
$a\in\Gamma(X)$,
\[
 w=m_X\circ(a\otimes j)-m_X\circ(j\otimes a).
\]
This is the infinitesimal conjugation direction $(0,a)$ in
\eqref{eq:joint-action}.  Hence the orbit-map differential is surjective
onto the tangent space of $\mathsf P_B(X)$.  Both factors of the group
are unit groups of finite-dimensional algebras and are smooth and
connected.  Lemma~\ref{lem:open-orbit} proves the assertion.
\end{proof}

The splitting-chart lemma itself holds over an arbitrary field and
requires only that $\End_{\C}(X)$ be finite-dimensional and that the
specified idempotent split.
The following chart uses the algebraic factor
$cb+(1-c)(1-b)$ in Kato's nearby-projection intertwiner \cite{Kato};
see also \cite[equation~(4)]{Simon}.

\begin{lemma}[A regular splitting chart]\label{lem:idempotent-chart}
Let $b\in\End_{\C}(X)$ be an idempotent with a splitting
$b=i\circ p$, $p\circ i=\id_B$.  For an idempotent $c$, put
\[
 t_c:=c\circ b+(\id_X-c)\circ(\id_X-b).
\]
Let $E=\End_{\C}(X)$ and let $L_t:E\to E$ denote left multiplication
by $t$.  In the variety $\operatorname{Idem}(E)=\{c\in E:c^2=c\}$, put
\[
 U_b:=\{c\in\operatorname{Idem}(E):\det(L_{t_c})\ne0\}.
\]
This is a Zariski-open neighborhood of $b$.  On $U_b$,
\[
 i_c:=t_c\circ i,\qquad p_c:=p\circ t_c^{-1}
\]
depend regularly on $c$ and satisfy
$p_c\circ i_c=\id_B$ and $i_c\circ p_c=c$.
\end{lemma}

\begin{proof}
We have $t_b=\id_X$ and
\[
 c\circ t_c=c\circ b=t_c\circ b.
\]
The entries of $L_{t_c}$ are affine-linear functions of the coordinates
of $c$, so its determinant is a polynomial, equal to $1$ at $b$.
An element of a finite-dimensional unital algebra is invertible exactly
when its left-multiplication map is invertible.  Thus $U_b$ is the
polynomial nonvanishing set on which $t_c$ is invertible; inversion is
regular there.  The displayed identity gives
$c=t_c\circ b\circ t_c^{-1}$ and hence both splitting identities.
For the exchange locus, the corresponding neighborhood is
$U_b\cap\Exch(X)$, open in its relative Zariski topology.
\end{proof}

\begin{lemma}[Open inner-conjugacy classes in the exchange locus]
\label{lem:exchange-open-classes}
Every point of $\Exch(X)$ corresponding to a separable Frobenius
subalgebra has an open neighborhood consisting of points in the same
inner-conjugacy class of underlying algebra subobjects.
\end{lemma}

\begin{proof}
Let $b=i\circ i^*$ correspond to a separable Frobenius subalgebra on
$B$.  Apply Lemma~\ref{lem:idempotent-chart} with $p=i^*$.
For $c\in\Exch(X)$ in this chart, put
\begin{equation}\label{eq:source-varying}
 e_c:=p_c\circ e_X,\qquad
 m_c:=p_c\circ m_X\circ(i_c\otimes i_c).
\end{equation}
The exchange correspondence makes $(m_c,e_c,i_c)$ a point of
$\mathsf P_B(X)$, and these formulas give a regular map from the chart
to that variety.  By Lemma~\ref{lem:joint-open-orbit}, the orbit of
$(m_b,e_b,i_b)$ is open.  Its inverse image is therefore an open
neighborhood of $b$.  If $c$ lies in this inverse image,
\eqref{eq:joint-action} supplies a source algebra isomorphism and an
ambient inner automorphism identifying the underlying algebra subobjects
at $b$ and $c$.  In particular, the source at $c$ is again separable.
\end{proof}

\begin{corollary}[Separable exchange idempotents are isolated]
\label{lem:exchange-isolated}
If $X$ is connected, every point of $\Exch(X)$ corresponding to a
separable Frobenius subalgebra is isolated.
\end{corollary}

\begin{proof}
Connectedness makes inner conjugation trivial.  By the exchange
correspondence and Lemma~\ref{lem:frobenius-unique}, each equivalence
class over $X$ corresponds to a single point.  The neighborhood in
Lemma~\ref{lem:exchange-open-classes} is therefore that single point.
\end{proof}

\subsection{The effective bound}

\begin{theorem}[Quadratic exchange bound]\label{thm:exchange-bound}
Let $\C$ be a Hom-finite $\kk$-linear monoidal category, and let $X$
be a Frobenius algebra in $\C$.
Let $N_F^{\mathrm{inn}}(X)$ be the number of classes of separable
Frobenius subalgebras under inner conjugacy of their underlying algebra
subobjects.  Then
\begin{align}
 N_F^{\mathrm{inn}}(X)&\leq2^{r_X}\leq2^{\dim_{\kk}\End_{\C}(X)},
 \label{eq:exchange-bound}\\
 2^{r_X}&\leq2^{\dim_{\kk}\End_{\C}(X)-1}\qquad(X\ne0).
 \label{eq:exchange-coarse}
\end{align}
If $X$ is connected, the same bound applies to the number $N_F(X)$ of
equivalence classes over $X$.  If, in addition, $\C$ is semisimple with
simple tensor unit and $X$ has the decomposition
\eqref{eq:target-decomposition}, then
\begin{equation}\label{eq:exchange-multiplicity}
 N_F(X)\leq2^{\sum_{i=1}^t n_i^2}.
\end{equation}
\end{theorem}

\begin{proof}
Passing to the Karoubi envelope preserves all spaces and operations
entering $r_X$, the algebra $\Gamma(X)$, and separability.  Its fully
faithful monoidal embedding also preserves the equivalence relation:
an isomorphism between original source objects in the envelope is
already a morphism of the original category.  Thus the original
inner-conjugacy classes inject into those in the envelope, and it
suffices to prove the bound when $\C$ is Karoubian.

The locus $\Exch(X)$ is defined by quadratics in the affine space
$\mathscr A_X$ of dimension $r_X$.  Theorem~\ref{thm:affine-bezout} gives
\[
 \cdeg\bigl(\Exch(X)\bigr)\leq2^{r_X}.
\]
By Lemma~\ref{lem:exchange-open-classes}, every separable inner-conjugacy
class is a nonempty open subset of $\Exch(X)$: take the union of the
neighborhoods supplied at its points.  Distinct such classes are disjoint.
Two nonempty open subsets meeting the same irreducible component must
intersect.  Consequently the number of these classes is at most the
number of irreducible components, and hence at most
$\cdeg(\Exch(X))$.  Since $r_X\leq\dim_{\kk}\End_{\C}(X)$,
this proves \eqref{eq:exchange-bound} uniformly.

For the refinement when $X\ne0$, let $E_X^+:=\{h\in\End_{\C}(X):h^*=h\}$.  The linear map
\[
 E_X^+\longrightarrow\Hom_{\C}(\one,X),\qquad h\longmapsto h\circ e_X
\]
sends $\id_X$ to $e_X\ne0$ and therefore has rank at least one.
Its kernel has dimension $r_X$, so
$r_X\leq\dim_{\kk}E_X^+-1\leq\dim_{\kk}\End_{\C}(X)-1$.
This proves \eqref{eq:exchange-coarse} without connectedness or a
simple-unit hypothesis.  Connectedness identifies inner conjugacy with
equivalence over $X$.  Finally,
$\dim_{\kk}\End_{\C}(X)=1+\sum_i n_i^2$ under
\eqref{eq:target-decomposition}, proving
\eqref{eq:exchange-multiplicity}.
\end{proof}

\begin{remark}[Why the base is two]\label{rem:why-base-two}
The direct multiplication-closure equation
\[
 b\circ m_X\circ(b\otimes b)=m_X\circ(b\otimes b)
\]
has degree three in the coordinates of $b$.  Applying the degree bound
to that presentation gives a base-three estimate.  Frobenius self-duality
and the exchange correspondence provide the quadratic presentation
\eqref{eq:exchange-locus}, yielding the base-two estimate.
\end{remark}

\begin{remark}\label{rem:split-quadratic-bound}
More generally, an algebra object $X$, without a Frobenius structure,
has at most $2^{d_X}$ inner-conjugacy classes of separable split algebra
subobjects, where
\[
 d_X=\dim_{\kk}\End_{\C}(X)-\dim_{\kk}\Hom_{\C}(\one,X).
\]
For connected $X$, this bounds their equivalence classes over $X$.
A compatible Frobenius structure gives $r_X\le d_X$, so
Theorem~\ref{thm:exchange-bound} can yield a sharper bound.
\end{remark}
\section{Application to Hopf algebras and finite groups}\label{sec:hopf}

We apply the exchange bound to the regular left comodule of a
finite-dimensional semisimple Hopf algebra, and then to function algebras
on finite homogeneous spaces.  Finite dimensionality makes the relevant
endomorphism spaces finite-dimensional; no unrestricted infinite-dimensional
extension is claimed.

\subsection{Semisimple Hopf algebras}

Let $H=(H,m_H,e_H,\delta_H,\epsilon_H,S)$ be a finite-dimensional
semisimple Hopf algebra over $\CC$.  Here $\Corep(H)$ denotes the
category of finite-dimensional left $H$-comodules.  A left coideal
subalgebra is a unital
subalgebra $K\subseteq H$ satisfying
\[
 \delta_H(K)\subseteq H\otimes K.
\]

Such coideal subalgebras are semisimple as ordinary algebras; see
\cite{Skryabin} and \cite[Lemma~4.0.2]{Burciu}, as recalled in
\cite{EtingofWalton}.  For the exchange bound we need
separability in $\Corep(H)$, established below using the Frobenius-subalgebra
correspondence and Lemma~\ref{lem:positive-frobenius-separable}.

\begin{theorem}[{\cite[Theorem~11.11, Proposition~6.5, and the proof of
Corollary~11.24]{GhoshPalcouxLattices}}]
\label{thm:ghosh-palcoux-hopf}
The following statements hold.
\begin{enumerate}[label=\textup{(\roman*)}]
  \item The regular comodule $H$ carries a natural connected Frobenius
  algebra structure
  $(H,m_H,e_H,\delta_H^{\mathrm F},\epsilon_H^{\mathrm F})$ in
  $\Corep(H)$.  Here the superscript distinguishes its Frobenius coalgebra
  maps from the Hopf maps $\delta_H$ and $\epsilon_H$.
  \item A unital subalgebra $K\subseteq H$ is a left coideal subalgebra if
  and only if it is a unital algebra subobject of $H$ in $\Corep(H)$.
  \item These unital algebra subobjects are precisely the Frobenius
  subalgebras of $H$ in $\Corep(H)$.
  \item The category $\Corep(H)$ is an integral, hence pseudo-unitary,
  fusion category and is therefore positive.
\end{enumerate}
\end{theorem}

Ostrik \cite{Ostrik} proves a splitting criterion for connected algebra
objects with a nondegenerate Frobenius pairing and nonzero pivotal dimension in
a pivotal semisimple setting.  For the corresponding splitting result for
connected Frobenius algebras of nonzero pivotal dimension in the ribbon
setting, see also \cite[Corollary~3.10 and Lemma~3.11]{FRS}.  We prove the
splitting conclusion below
without a semisimplicity assumption.

\begin{lemma}[Connected Frobenius algebras of nonzero pivotal dimension]
\label{lem:pivotal-frobenius-separable}
Let $\kk$ be any field and let $\C$ be a $\kk$-linear pivotal monoidal
category with pivotal structure $\phi$.  A connected Frobenius algebra $A$ in $\C$ is separable whenever
its pivotal dimension $\dim_{\phi}(A)\in\End_{\C}(\one)$ is nonzero.
\end{lemma}

\begin{proof}
Write $A=(A,m_A,e_A,\delta_A,\epsilon_A)$.  Using the Frobenius
self-duality \eqref{eq:frobenius-pairing-coeval}, identify $A^{**}$ with
$A$ and write $\alpha$ for $\phi_A$ under this identification.
Naturality and monoidality of $\phi$ make $\alpha$ an algebra
automorphism.  Set $T=m_A\circ(\alpha\otimes\id_A)\circ\delta_A$.
Connectedness gives $T\circ e_A=c e_A$ for some $c\in\kk$.
The Frobenius identities give
\[
 T=m_A\circ((T\circ e_A)\otimes\id_A)=c\,\id_A,
 \qquad
 T\circ m_A=m_A\circ(\alpha\otimes T).
\]
Inserting the unit in the second identity gives $T=c\alpha$.
By the definition of pivotal dimension,
\[
 \dim_{\phi}(A)
 =\ev_A\circ(\alpha\otimes\id_A)\circ\coev_A
 =\epsilon_A\circ T\circ e_A
 =c(\epsilon_A\circ e_A).
\]
Thus $c\ne0$, so $\alpha=\id_A$ and
$m_A\circ\delta_A=c\,\id_A$.  Since $\delta_A$ is an
$A$--$A$-bimodule morphism, $c^{-1}\delta_A$ splits $m_A$.
\end{proof}

Following \cite[Definition~6.3]{GhoshPalcouxLattices}, positivity means that
the category admits a pivotal structure $\phi$ for which the pivotal
dimension $\dim_{\phi}(Y)$ is positive for every nonzero object $Y$.

\Needspace{6\baselineskip}
\begin{lemma}[Connected Frobenius algebras in positive tensor categories]
\label{lem:positive-frobenius-separable}
Every connected Frobenius algebra in a positive tensor category is
separable.
\end{lemma}

\begin{proof}
For such an algebra $A$, connectedness gives $A\ne0$, and a positive
pivotal structure gives $\dim_{\phi}(A)>0$; apply
Lemma~\ref{lem:pivotal-frobenius-separable}.
\end{proof}

The ambient endomorphism dimension is particularly simple.

\begin{lemma}\label{lem:regular-endomorphisms}
For any finite-dimensional Hopf algebra $H$ over any field $\kk$,
the regular left comodule satisfies
\[
 \End_{\Corep(H)}(H)\cong H^*
\]
as vector spaces.  In particular,
$\dim_{\kk}\End_{\Corep(H)}(H)=\dim_{\kk}H$.
\end{lemma}

\begin{proof}
Every $\varphi\in H^*$ defines a comodule endomorphism
$(\id_H\otimes\varphi)\circ\delta_H$.  Conversely, if $T$ is a comodule
endomorphism, then
\[
 T=(\id_H\otimes(\epsilon_H\circ T))\circ\delta_H.
\]
These assignments are inverse.
\end{proof}

\Needspace{8\baselineskip}
\begin{corollary}[Semisimple Hopf algebras]\label{cor:hopf-bound}
Let $H$ be a semisimple Hopf algebra over $\CC$ with
$\dim_{\CC}H=N$.  Then $H$ has at most $2^{N-1}$ left coideal subalgebras.
\end{corollary}

\begin{proof}
By Theorem~\ref{thm:ghosh-palcoux-hopf}\textup{(i),(iv)}, $H$ is a
connected Frobenius algebra in the positive tensor category $\Corep(H)$.
By Theorem~\ref{thm:ghosh-palcoux-hopf}\textup{(ii),(iii)}, every left
coideal subalgebra $K$ is a Frobenius subalgebra.  Its monomorphism
$i:K\to H$ injects $\Hom_{\Corep(H)}(\one,K)$ into
$\Hom_{\Corep(H)}(\one,H)=\CC e_H$, while the former space contains the
nonzero morphism $e_K$; hence $K$ is connected.  Lemma~\ref{lem:positive-frobenius-separable}
therefore makes every such $K$ separable.  Theorem~\ref{thm:exchange-bound}
and Lemma~\ref{lem:regular-endomorphisms} now give the stated bound.
\end{proof}

This is an effective strengthening of the qualitative finiteness theorem of
Etingof--Walton \cite[Theorem~3.6]{EtingofWalton}.
A distinct earlier rigidity-to-finiteness theorem is \c{S}tefan's finiteness
of isomorphism types of semisimple and cosemisimple Hopf algebras of fixed
dimension over $\CC$ \cite{Stefan}.

\subsection{Intermediate subgroups as a check}

The group case gives a transparent specialization.  Let $G$ be finite and
$H\leq G$.  The commutative Frobenius algebra
$X=\operatorname{Fun}(G/H)$ in $\Rep(G)$ is connected, and its
$G$-stable unital subalgebras correspond to intermediate subgroups
$H\leq K\leq G$.  By Frobenius reciprocity and the Mackey restriction
formula \cite[\S\S7.2--7.3]{SerreLinearRepresentations},
\begin{equation}\label{eq:double-coset-end}
 \dim_{\CC}\End_{\Rep(G)}\bigl(\operatorname{Fun}(G/H)\bigr)
 =|H\backslash G/H|.
\end{equation}

\begin{corollary}[Double-coset bound]\label{cor:double-coset}
For every finite group $G$ and subgroup $H\leq G$,
\begin{equation}\label{eq:double-coset-bound}
 \bigl|\{K:H\leq K\leq G\}\bigr|
 \leq 2^{|H\backslash G/H|-1}.
\end{equation}
\end{corollary}

\begin{proof}
A $G$-stable subalgebra of $\operatorname{Fun}(G/H)$ is the algebra of
functions constant on the blocks of a $G$-invariant partition of $G/H$;
such partitions are precisely the quotient maps $G/H\to G/K$ with
$H\leq K\leq G$.  Equip each function algebra on a homogeneous space
$S$ with the averaged counit
$\epsilon_S(f)=|S|^{-1}\sum_{s\in S}f(s)$ and comultiplication
$\delta_S(p_s)=|S|p_s\otimes p_s$, where $p_s$ is the characteristic
function of $s$.  For pullback $j$ along $G/H\to G/K$, the Frobenius mate
$j^*$ is fibrewise averaging, so $j^*j=\id$; the symmetric pairings also
give $j^{**}=j$.  Thus $\operatorname{Fun}(G/K)$ is a separable Frobenius
subalgebra of $\operatorname{Fun}(G/H)$.  The endomorphism dimension in
\eqref{eq:double-coset-end} is the number of $G$-orbits on
$(G/H)\times(G/H)$, namely the number of double cosets.  Apply
\eqref{eq:exchange-coarse}.
\end{proof}

The same estimate also follows directly: every intermediate subgroup is a
union of $H$-double cosets containing $H$, leaving at most
$2^{|H\backslash G/H|-1}$ possibilities.  The categorical derivation is a
useful consistency check: the dimension appearing in the exponent becomes
this familiar permutation-theoretic invariant.

\begin{theorem}[Sharp finite-group interval bound; Theorem~1.1 and
Corollary~1.2 of \cite{PalcouxSharpIntervals}]
\label{thm:sharp-group-interval-bound}
Let $H\leq G$ be finite groups, put $n=[G:H]>1$, and let $p$ be the least
prime divisor of $n$.  Set
\[
 c(p):=
 \left(\prod_{j=1}^{\infty}(1-p^{-j})^{-1}\right)
 \left(\sum_{z\in\mathbb Z}p^{-z^2}\right).
\]
Then
\[
 \bigl|\{K:H\leq K\leq G\}\bigr|
 <c(p)n^{\frac14\log_p n}.
\]
In particular,
\[
 \bigl|\{K:H\leq K\leq G\}\bigr|
 <c(2)n^{\frac14\log_2 n}
 <7.371968802\,n^{\frac14\log_2 n}.
\]
The coefficient $\frac14$ in the exponent and, with this coefficient fixed,
the constant $c(p)$ are optimal among inclusions whose index has least
prime divisor $p$: for $H=1$ and $G=(\mathbb Z/p\mathbb Z)^{2a}$,
\[
 \frac{\bigl|\{K:1\leq K\leq G\}\bigr|}
 {|G|^{\frac14\log_p|G|}}\longrightarrow c(p)
 \qquad(a\longrightarrow\infty).
\]
\end{theorem}

This index-only estimate is complementary to
Corollary~\ref{cor:double-coset}: it is asymptotically much smaller than the
worst-case bound $2^{n-1}$, whereas the double-coset bound can be sharper for
a particular action when $|H\backslash G/H|$ is small.  For $H=1$ it
specializes to the sharp absolute estimate of Fusari--Spiga
\cite{FusariSpiga}; it is its relative analogue and improves
Guerra--Mastrogiacomo--Spiga \cite{GuerraMastrogiacomoSpiga} by removing the
factor $n^{1.8919}$.  The proof in \cite{PalcouxSharpIntervals} counts relative generating
tuples by their subgroup index and then estimates a discrete Gaussian sum.
It is elementary and avoids the classification of finite simple groups,
suggesting a possible route to categorical refinements.

\section{Applications to unitary tensor categories and operator-algebra
inclusions}\label{sec:subfactors}

We first specialize the general rigidity results to unitary Frobenius
algebras, and then apply them to finite-index operator-algebra inclusions.

\subsection{Unitary Frobenius algebras}\label{sec:unitary-frobenius}

Let $\C$ be a unitary tensor category, with possibly infinitely many
simple isomorphism classes; see \cite[Remark~9.4.7]{EGNO} for the fusion
case.  A \emph{normalized unitary Frobenius algebra}
is a connected Frobenius algebra
$(B,m_B,e_B,\delta_B,\epsilon_B)$ such that
\[
\begin{gathered}
  \delta_B=m_B^\dagger,\qquad \epsilon_B=e_B^\dagger,\qquad
  e_B^\dagger\circ e_B=\id_{\one},\\
  m_B\circ\delta_B=\lambda\id_B\quad(\lambda>0).
\end{gathered}
\]
The isometric-unit identity is the normalization; the last identity is
specialness, with strictly positive scalar.  These are the normalized
connected $Q$-system conventions of
\cite[Definitions~2.12 and~2.16--2.17]{CGGH}, expressed
in Frobenius-algebra terminology; see also
\cite{Longo,BKLR,JonesPenneys,JonesPenneysQ}.  A morphism
$u:B\to B'$ is unitary if $u^\dagger\circ u=\id_B$ and
$u\circ u^\dagger=\id_{B'}$.

Specialness gives the bimodule section and separability morphism
\begin{equation}\label{eq:unitary-section}
  s_B:=\lambda^{-1}\delta_B,\qquad
  p_B:=s_B\circ e_B,\qquad
  m_B\circ s_B=\id_B.
\end{equation}
The Frobenius identities make $s_B$ a $B$--$B$-bimodule morphism.  Thus
every normalized unitary Frobenius algebra is separable in the sense of
\S{}\ref{subsec:bimodule-separability}.

More generally, in the unitary multitensor setting, an algebra is separable
if and only if it is isomorphic to a Frobenius algebra whose structure maps
satisfy $\delta=m^\dagger$, $\epsilon=e^\dagger$, and
$m\circ m^\dagger=\id$ \cite[Theorem~4.13]{GiorgettiYuanZhao};
connectedness is not required.  This normalizes the multiplication, and
does not impose the isometric-unit normalization above.
For the normalized connected algebras considered here, every algebra
isomorphism is unitary by \cite[Theorem~3.9]{CGGH}.  The algebraic rigidity
results therefore give the following.

\begin{corollary}[Fixed-object unitary finiteness]
\label{cor:fixed-unitary-frobenius}
Every object $B$ of $\C$ supports only finitely many unitary-equivalence
classes of normalized unitary Frobenius algebra structures.
\end{corollary}

\begin{proof}
Every such algebra is separable by \eqref{eq:unitary-section}.  Apply
Theorem~\ref{thm:fixed-object}, followed by
\cite[Theorem~3.9]{CGGH}.
\end{proof}

\begin{corollary}[Unitary Frobenius subalgebras]
\label{cor:unitary-frobenius-subalgebra}
A connected normalized unitary Frobenius algebra $X$ in $\C$ has only
finitely many normalized unitary Frobenius subalgebras up to equivalence
over $X$.
Moreover, their number is at most
\[
  2^{r_X}\leq 2^{\dim_{\CC}\End_{\C}(X)-1}.
\]
\end{corollary}

\begin{proof}
The ambient algebra $X$ is connected, and every normalized unitary
Frobenius subalgebra is separable by \eqref{eq:unitary-section}.
Apply Corollary~\ref{cor:frobenius} and
Theorem~\ref{thm:exchange-bound}.  Any algebra isomorphism realizing
equivalence over $X$ is itself unitary by \cite[Theorem~3.9]{CGGH},
so the equivalence still commutes with the embeddings into $X$.
\end{proof}

\subsection{\texorpdfstring{Finite-index $C^*$-algebra inclusions}{Finite-index C*-algebra inclusions}}

Throughout this subsection, finite index means \emph{finite Watatani
index}: the faithful conditional expectation (a norm-one projection)
under consideration admits a finite quasi-basis.  Let $A\subseteq C$ be a unital inclusion with such an
expectation $E:C\to A$, and assume that $Z(A)$ is finite-dimensional.
Equip $C_A$ with the inner product $\langle x,y\rangle_A=E(x^*y)$,
let $C_1$ be the basic construction, and put
$z_E:=\operatorname{Ind}(E)\in Z(C)_+^\times$.
We distinguish this index from the Pimsner--Popa index; see
\cite[Example~3.10 and the paragraph following Lemma~3.11]{CHPJP}
and \cite{WatataniIndex}.

An intermediate unital $C^*$-algebra $A\subseteq B\subseteq C$ is
\emph{$E$-compatible} if there is a conditional expectation $F_B:C\to B$
such that $E=(E|_B)F_B$
\cite[Definition~10.4]{GhoshPalcouxLattices}.  Both endpoints are included.
Two compatible intermediates $B,D$ are \emph{unitarily conjugate} if
$D=uBu^*$ for some unitary $u\in A'\cap C$.  This compares compatible
intermediates without assuming that every such conjugation preserves
$E$-compatibility.

The analytic local-conjugacy approach to intermediate $C^*$-algebras is
developed in \cite{InoWatatani} and \cite[Theorem~3.7]{Dickson}.
Under their respective hypotheses, sufficiently close intermediates are
conjugate by a unitary in the relative commutant; here the counting argument
proceeds through algebraic local rigidity and the exchange locus.

Write $C^*\!\operatorname{Alg}(A,A)$ for the monoidal category of
$A$--$A$ correspondences: right Hilbert $A$-modules with a unital left
$A$-action by adjointable operators, with adjointable bimodule maps and
interior tensor product over $A$ \cite[Definition~2.12]{CHPJP}.
Let $\C$ be its full subcategory of dualizable correspondences, with
finite direct sums and split projections.
Since $\End({}_AA_A)=Z(A)$ is finite-dimensional, the conjugate equations
give Hom-finiteness; splitting minimal projections gives semisimplicity
and finite length.  This is the unitary multitensor setting of
\cite[Lemma~3.11 and Remark~3.13]{CHPJP}.

By the finite-index construction in \cite[Lemma~3.11]{CHPJP}, the correspondence $X={}_AC_A$,
with its usual multiplication $m$ and unit $e$, is a dagger Frobenius
algebra.  Its structure maps satisfy
\begin{equation}\label{eq:operator-frobenius-data}
 e^\dagger=E,\qquad m\circ m^\dagger=L_{z_E},
\end{equation}
where $L_x$ denotes left multiplication by $x$; thus the second identity
is an equality of endomorphisms of $X$, with
$m^\dagger:X\to X\otimes_A X$.
The central positive invertibility of $z_E$ makes
$m^\dagger L_{z_E^{-1}}$ a bimodule section of $m$, so $X$ is separable.

The inclusion correspondence $Y:={}_AC_C$ has dual
$Y^\vee:={}_CC_A$, and $X\cong Y\otimes_CY^\vee$
\cite[Example~3.25]{CHPJP}.  The duality adjunction for these correspondences,
followed by $T\mapsto T(1_C)$ on endomorphisms of $Y$, gives
\begin{equation}\label{eq:operator-connected-relative}
 \Hom_{A\text{-}A}(A,X)
 \cong\Hom_{A\text{-}A}(A,Y\otimes_CY^\vee)
 \cong\End_{A\text{-}C}(Y)
 \cong A'\cap C.
\end{equation}
Here $1_C=1_A$ is the algebra unit; the tensor unit of $\C$ is the
correspondence ${}_AA_A$.  Explicitly, a bimodule map $f:A\to C$
corresponds to $c\mapsto f(1_A)c$, while $x\in A'\cap C$ corresponds
to $f_x(a)=ax$.  Under this identification $\Gamma(X)$ has the ordinary
multiplication of $A'\cap C$.  Thus irreducibility is exactly connectedness.
\Needspace{12\baselineskip}
Moreover,
\begin{equation}\label{eq:operator-end-relative}
 \End_{\C}(X)=A'\cap C_1.
\end{equation}
Indeed, let $(u_i)$ be a finite right frame of $C_A$ and write
$e_A(c)=E(c)$.  For every $T\in\mathcal L_A(C_A)$,
\[
 T=\sum_i\theta_{Tu_i,u_i},\qquad
 \theta_{x,y}(c)=xE(y^*c),\qquad
 \theta_{x,y}=L_xe_AL_{y^*}.
\]
Hence $\mathcal L_A(C_A)=\mathcal K_A(C_A)=C_1$.
These operators are already right $A$-linear; they are bimodule maps
exactly when they commute with all $L_a$, $a\in A$, proving
\eqref{eq:operator-end-relative}.

Together with $Z(A)\subseteq A'\cap C\subseteq A'\cap C_1$, this shows
that, under finite Watatani index, finite-dimensionality of $Z(A)$,
of $A'\cap C$, and of $A'\cap C_1$ are equivalent conditions.
Finite index alone implies none of them, as the index-two example below
shows.

If $B$ is $E$-compatible, its inclusion $j:{}_AB_A\to X$ is
adjointable, with $j^\dagger=F_B$
\cite[Proposition~10.5]{GhoshPalcouxLattices}.  Thus $B_A$ is a direct
summand of $C_A$, hence finitely generated projective.  The restricted
expectation makes ${}_AB_A$ a separable dagger Frobenius algebra.
Compatibility and bimodularity identify the Frobenius transpose of $j$
with $F_B$, so $j^*j=\id$ and $j^{**}=j$.  Consequently every
$E$-compatible intermediate gives a separable Frobenius subalgebra of $X$.

\begin{theorem}[General $C^*$-algebra bound]
\label{thm:cstar-general-bound}
Let $A\subseteq C$ be a unital inclusion with a faithful conditional
expectation $E$ of finite Watatani index, and suppose that $Z(A)$ is
finite-dimensional.  The number of $E$-compatible intermediate unital
$C^*$-algebras, up to conjugation by unitaries in $A'\cap C$, is at most
\[
 2^{\dim_{\CC}(A'\cap C_1)-1}.
\]
If $Z(A)=\CC1$, this bound is at most
\[
 2^{\lfloor\|\operatorname{Ind}(E)\|^2\rfloor-1}
 <2^{\|\operatorname{Ind}(E)\|^2}.
\]
\end{theorem}

\begin{proof}
Theorem~\ref{thm:exchange-bound} and
\eqref{eq:operator-end-relative} give $2^{\dim_{\CC}(A'\cap C_1)-1}$ for conjugacy by
invertible elements of $A'\cap C$.
Invertible and unitary conjugacy agree for these subalgebras.  Indeed,
if $uBu^{-1}=D$, self-adjointness gives $hBh^{-1}=B$, where $h=u^*u$.
The relative commutant is finite-dimensional, so $h$ has finite spectrum.
Polynomial interpolation on the finitely many spectral ratios of $h$
expresses $\operatorname{Ad}_{h^{-1/2}}$ as a polynomial in
$\operatorname{Ad}_h$.  Hence $h^{-1/2}$ normalizes $B$, and the unitary
$uh^{-1/2}$ still sends $B$ onto $D$.

If $Z(A)=\CC1$, the Frobenius cup $R=m^\dagger e$ is a selfdual
conjugate solution for $X$, and
\begin{equation}\label{eq:operator-dimension-index-upper}
 R^\dagger R=E(z_E),\qquad
 d(X)\leq\|E(z_E)\|\leq\|\operatorname{Ind}(E)\|.
\end{equation}
Writing $X=\bigoplus_iS_i^{\oplus n_i}$ and using $d(S_i)\geq1$ gives
\[
 \dim_{\CC}(A'\cap C_1)=\sum_i n_i^2\leq\left(\sum_i n_i\right)^2
 \leq d(X)^2\leq\|\operatorname{Ind}(E)\|^2.
\]
Since $\dim_{\CC}(A'\cap C_1)$ is an integer, this proves the second bound.
\end{proof}

Qualitative finiteness up to relative-commutant unitary conjugacy can also
be extracted from the proof of \cite[Theorem~4.2]{GuptaKumar}.
The compatible Jones projections lie in the compact projection space of
the finite-dimensional algebra $A'\cap C_1$; a finite cover by sufficiently
small balls, together with the perturbation theorem, leaves only finitely
many conjugacy classes.  Their additional condition on the intermediates
is used only to conclude equality from conjugacy.
Theorem~\ref{thm:cstar-general-bound} supplies
the explicit exchange-locus bound.

\begin{corollary}[Irreducible $C^*$-algebra bound]
\label{cor:cstar-irreducible-bound}
Let $A\subseteq C$ be a unital irreducible inclusion with a faithful
conditional expectation $E$ of finite Watatani index.  Its Watatani
index is scalar, and the number of $E$-compatible intermediate unital
$C^*$-algebras is at most
\[
 2^{\dim_{\CC}(A'\cap C_1)-1}
 \leq 2^{\lfloor\|\operatorname{Ind}(E)\|\rfloor-1}
 <2^{\|\operatorname{Ind}(E)\|}.
\]
Neither algebra is required to be simple.
\end{corollary}

\begin{proof}
Irreducibility implies $Z(A)=Z(C)=\CC1$ and makes $X$ connected by
\eqref{eq:operator-connected-relative}.  Inner conjugation is trivial.
The dimension comparison in Proposition~\ref{prop:dimension-comparison}
and \eqref{eq:operator-dimension-index-upper} give
$\dim_{\CC}(A'\cap C_1)\leq d(X)\leq\|\operatorname{Ind}(E)\|$.
\end{proof}

The qualitative finiteness in this corollary was proved by Gupta--Kumar
\cite[Corollary~4.3]{GuptaKumar}; see also
\cite[Corollary~10.6]{GhoshPalcouxLattices}.  The restriction to
$E$-compatible intermediates is essential: Ghosh--Palcoux give an
irreducible finite-index inclusion with uncountably many unrestricted
intermediate $C^*$-algebras \cite[Example~10.7]{GhoshPalcouxLattices}.

For irreducible inclusions $A\subseteq C$ of simple unital $C^*$-algebras
with finite minimal Watatani index $J$, earlier whole-lattice bounds are
$\min\{9^{J^2},(J^2)^{J^2}\}$ in \cite{BakshiGupta}
and $9^J$ in \cite{BakshiGuinJana}.
With the minimal expectation and $D:=\dim_{\CC}(A'\cap C_1)$, the latter
proof obtains $9^D$ before applying $D\leq J$.
Here the finite-index expectation is unique and minimal, and all
intermediates are compatible with it; on this common class of inclusions,
the present estimate is $2^{D-1}$.

The finite-center hypothesis in the general theorem also has content.
For the diagonal inclusion $C(K)\subseteq C(K)\oplus C(K)$ with the
averaging expectation, the Watatani index is $2$, whereas
$A'\cap C=C(K)\oplus C(K)$ and
$A'\cap C_1=C_1\cong M_2(C(K))$.  Thus neither relative commutant
is finite-dimensional when $K$ is infinite.  Every clopen
$U\subseteq K$ gives an $E$-compatible intermediate
\[
 B_U=\{(f,g):f=g\text{ on }K\setminus U\}.
\]
If $K$ is the Cantor set, these give infinitely many distinct classes,
since the ambient algebra is commutative.  If $K$ has $n$ points, they
give $2^n$ classes at the same index $2$.

\begin{remark}[Von Neumann algebra inclusions]
\label{rem:von-neumann-inclusions}
The preceding $C^*$-algebra bounds apply to the $E$-compatible
intermediate von Neumann algebras of a unital inclusion
$N\subseteq M$ equipped with a faithful normal conditional
expectation $E:M\to N$ of finite Watatani index, with $Z(N)$
finite-dimensional.  Indeed, these intermediates belong to the
collection of intermediate $C^*$-algebras already counted.
The unitary inner-conjugacy convention is unchanged, and no
factoriality assumption is imposed.

If $M$ admits a faithful normal tracial state $\tau$ and $E$
is $\tau$-preserving, every intermediate von Neumann algebra
$P$ is $E$-compatible: its $\tau$-preserving conditional
expectation $E_P:M\to P$ satisfies $(E|_P)\circ E_P=E$, by
\cite[Theorem~9.1.2 and Lemma~9.1.4(v)]{AnantharamanPopa}.

Finally, an irreducible inclusion of von Neumann algebras is
already an inclusion of factors, since
$Z(N),Z(M)\subseteq N'\cap M=\CC1$.
\end{remark}

\subsection{Finite-index subfactors}

Let $N\subseteq M$ be a finite-index inclusion of $\mathrm{II}_1$ factors,
let $I:=[M:N]$ be its Jones index, and let $E_N:M\to N$ be the
trace-preserving conditional expectation.  Its Watatani index is $I1_M$.
Every intermediate von Neumann algebra $N\subseteq P\subseteq M$ is
$E_N$-compatible, because its trace-preserving expectation satisfies
$E_N=E_N|_P\circ E_P$.  Such an intermediate need not be a factor when
$N\subseteq M$ is reducible.

Let $\mathcal L(N\subseteq M)$ denote all these intermediate von
Neumann algebras, including $N$ and $M$, and let
$M_1:=\langle M,e_N\rangle$ be the basic construction.
Let $\C$ now denote the unitary tensor category of dualizable
$N$--$N$ correspondences.  The correspondence realization above gives
\begin{equation}\label{eq:canonical-unitary-frobenius}
 H={}_NM_M,\qquad \overline H={}_MM_N,\qquad
 X=H\otimes_M\overline H\cong{}_NM_N,
\end{equation}
and the identities
\begin{equation}\label{eq:end-relative-commutant}
 \Hom_{\C}(\one,X)\cong\End_{N\text{-}M}(H)\cong N'\cap M,
 \qquad\End_{\C}(X)\cong N'\cap M_1.
\end{equation}
In particular, irreducibility is exactly connectedness.  The cup estimate
\eqref{eq:operator-dimension-index-upper} gives $d(X)\leq I$; this is all
that is needed below, with no standardness assumption on this chosen
Frobenius duality.

\begin{theorem}[General intermediate-algebra bound]
\label{thm:subfactor-general-bound}
For a finite-index inclusion $N\subseteq M$ of $\mathrm{II}_1$ factors,
the number of intermediate von Neumann algebras, up to conjugation by
unitaries in $N'\cap M$, is at most
\[
 2^{\dim_{\CC}(N'\cap M_1)-1}
 \leq2^{\lfloor[M:N]^2\rfloor-1}<2^{[M:N]^2}.
\]
The same bounds hold for intermediate subfactors up to this conjugacy.
\end{theorem}

\begin{proof}
Apply Theorem~\ref{thm:cstar-general-bound} to the inclusion with its
trace-preserving expectation, and restrict to intermediate von Neumann
algebras.  In this case conjugation by $\mathcal U(N'\cap M)$ preserves
the trace and the set of compatible intermediates.
\end{proof}

The general theorem gives an explicit bound on conjugacy classes.
Related finiteness results for distinguished families appear in
\cite{KhoshkamMashood}.  For a finite von Neumann algebra inclusion with
a faithful normal tracial state, its trace-preserving expectation of finite
Watatani index, and finite-dimensional
center of the smaller algebra, \cite[Theorem~6.4]{BakshiGupta} proves
finiteness of the actual intermediate lattice when the relative commutant
equals the center of either algebra; Theorem~\ref{thm:subfactor-general-bound}
counts conjugacy classes without these extra equalities.  Under
irreducibility, conjugation becomes trivial, recovering the qualitative
finiteness theorem of \cite{WatataniLattices}.

\subsection{Endomorphism dimension versus categorical dimension}

For irreducible simple $C^*$-inclusions $A\subseteq C$, the estimate
$\dim_{\CC}(A'\cap C_1)\leq J$, with $J$ the minimal Watatani index, is already
used in the counting argument of
\cite{BakshiGuinJana}.  The following argument uses the
module-category viewpoint of \cite{Ostrik} and gives the categorical
endomorphism-dimension estimate in the positive semisimple setting.

\begin{proposition}[Dimension comparison]\label{prop:dimension-comparison}
Let $\C$ be a positive semisimple tensor category over $\CC$, in the sense
of \cite[Definition~6.3]{GhoshPalcouxLattices}, with positive pivotal
dimension $d$, and let $A$ be a connected finite-length separable algebra
object in $\C$.  Write
\begin{equation}\label{eq:algebra-decomposition-dimension}
 A\cong\one\oplus\bigoplus_{i=1}^t S_i^{\oplus n_i},
\end{equation}
where the $S_i\not\cong\one$ are pairwise nonisomorphic simple objects.
Then
\begin{equation}\label{eq:multiplicity-dimension}
 n_i\leq d(S_i)\qquad(1\leq i\leq t),
\end{equation}
and consequently
\begin{equation}\label{eq:end-less-dimension}
 \dim_{\CC}\End_{\C}(A)\leq d(A).
\end{equation}
More precisely,
\begin{equation}\label{eq:categorical-defect}
 d(A)-\dim_{\CC}\End_{\C}(A)
 =\sum_{i=1}^t n_i\bigl(d(S_i)-n_i\bigr).
\end{equation}
\end{proposition}

\begin{proof}
Let $\C_A$ be the category of right $A$-modules.  Choose a separability
morphism $p_A:\one\to A\otimes A$.  If $q:M\to N$ is a module map and
$t:N\to M$ is a splitting in $\C$, let $r_Y:Y\otimes A\to Y$ denote the
action on a right module $Y$ and set
\[
  \sigma_Y:=(r_Y\otimes\id_A)\circ(\id_Y\otimes p_A),\qquad
  \widetilde t:=r_M\circ(t\otimes\id_A)\circ\sigma_N.
\]
Centrality of $p_A$ makes $\sigma_Y$ and $\widetilde t$ right $A$-linear,
while normalization gives
$r_Y\circ\sigma_Y=\id_Y$ and $q\circ\widetilde t=\id_N$.  Thus splittings
can be averaged to module splittings, and $\C_A$ is semisimple.  The regular
right module $A_A$ is simple, because
\[
 \End_{\C_A}(A_A)\cong\Hom_{\C}(\one,A)=\CC
\]
by connectedness.  For each $i$, the free--forgetful adjunction gives
\[
 \dim_{\CC}\Hom_{\C_A}(S_i\otimes A,A)
 =\dim_{\CC}\Hom_{\C}(S_i,A)=n_i.
\]
Hence, in the semisimple module category,
\[
 S_i\otimes A\cong A^{\oplus n_i}\oplus Y_i
\]
for some right $A$-module $Y_i$.  Applying the positive pivotal dimension
to the underlying objects gives
\[
 d(S_i)d(A)=n_i d(A)+d(Y_i),
\]
so $n_i\leq d(S_i)$.  Finally,
\[
 d(A)=1+\sum_i n_i d(S_i),
 \qquad
 \dim_{\CC}\End_{\C}(A)=1+\sum_i n_i^2,
\]
and subtraction proves \eqref{eq:end-less-dimension} and
\eqref{eq:categorical-defect}.
\end{proof}

For the remaining subfactor estimates in this section, assume that
$N\subseteq M$ is irreducible.  Then the inclusion correspondence $H$
is simple.  A balanced conjugate solution for a simple correspondence
is standard; the dimension--index identity therefore gives
\begin{equation}\label{eq:unitary-frobenius-index}
 d(X)=d(H)^2=[M:N]=I.
\end{equation}
Here the Jones and minimal indices agree; see \cite[\S2.2]{BKLR}.
Applying Proposition~\ref{prop:dimension-comparison} to the connected
unitary Frobenius algebra $X={}_NM_N$ gives
\begin{equation}\label{eq:relative-commutant-index}
 \dim_{\CC}(N'\cap M_1)
 =\dim_{\CC}\End_{\C}(X)
 \leq d(X)=[M:N].
\end{equation}
The proof also makes the difference transparent.  If
\begin{equation}\label{eq:canonical-decomposition}
 X\cong{}_NN_N\oplus\bigoplus_{\alpha}V_\alpha^{\oplus n_\alpha},
\end{equation}
then
\begin{equation}\label{eq:subfactor-defect}
 [M:N]-\dim_{\CC}(N'\cap M_1)
 =\sum_\alpha n_\alpha\bigl(d(V_\alpha)-n_\alpha\bigr).
\end{equation}
Indeed, bimodule Frobenius reciprocity gives
\[
 n_\alpha
 =\dim_{\CC}\Hom_{\C}(V_\alpha,X)
 =\dim_{\CC}\Hom_{N\text{-}M}(V_\alpha\otimes_N H,H).
\]
Irreducibility gives
$\End_{N\text{-}M}(H)\cong N'\cap M=\CC$, so $H$ is simple.  Thus
\[
 V_\alpha\otimes_N H\cong H^{\oplus n_\alpha}\oplus W_\alpha
\]
and
\begin{equation}\label{eq:depth-three-defect}
 d(V_\alpha)-n_\alpha
 =\frac{d(W_\alpha)}{d(H)}\geq0.
\end{equation}
In principal-graph language, $W_\alpha$ records the other odd vertices
adjacent to $V_\alpha$.  Thus \eqref{eq:subfactor-defect} records branching
beyond the distinguished depth-one vertex.  In particular, equality holds in
\eqref{eq:relative-commutant-index} for depth-two inclusions.

\subsection{A direct bound for the intermediate lattice}

Under the standing irreducibility assumption, $\mathcal L(N\subset M)$
is the lattice of intermediate subfactors, including $N$ and $M$.

\begin{theorem}[Irreducible intermediate-subfactor bound]
\label{thm:subfactor-bound}
Let $N\subseteq M$ be a finite-index irreducible $\mathrm{II}_1$
subfactor.  Then
\begin{equation}\label{eq:subfactor-bound-chain}
 |\mathcal L(N\subset M)|
 \leq2^{r_X}\leq2^{\dim_{\CC}(N'\cap M_1)-1}
 \leq2^{\lfloor[M:N]\rfloor-1}<2^{[M:N]}.
\end{equation}
\end{theorem}

\begin{proof}
Irreducibility makes $X$ connected and the conjugation action trivial.
It also forces every intermediate von Neumann algebra to be a factor,
since $Z(P)\subseteq N'\cap M=\CC1$.
Theorem~\ref{thm:exchange-bound} gives the first two inequalities.
Proposition~\ref{prop:dimension-comparison} gives
$\dim_{\CC}(N'\cap M_1)\leq d(X)\leq I$, proving the rest.
\end{proof}

The relative-commutant form is considerably sharper than its index-only
corollary whenever $X$ has few simple summands.  For
example, a proper $2$-supertransitive inclusion has
$\dim_{\CC}(N'\cap M_1)=2$, and
Theorem~\ref{thm:subfactor-bound} gives
$|\mathcal L(N\subset M)|\leq2$; hence there is no nontrivial intermediate
subfactor.  At the opposite extreme, for a depth-two inclusion the difference
\eqref{eq:subfactor-defect} vanishes, so replacing the relative-commutant
dimension by the index loses no information.

\subsection{Comparison of intermediate-subfactor bounds}
\label{subsec:subfactor-comparison}

For earlier qualitative finiteness results, see Watatani
\cite{WatataniLattices} and the type~III treatment \cite{TeruyaWatatani}.
Longo \cite{Longo2003} gives a
unified argument and explicit bounds; for a further bound on irreducible
$\sigma$-finite type~III inclusions of finite Watatani index, see
\cite{BakshiGupta}.

For the numerical comparison below, $N\subset M$ is a finite-index irreducible
$\mathrm{II}_1$ subfactor.  Set
\[
 I:=[M:N],
 \qquad
 D:=\dim_{\CC}(N'\cap M_1)
   =\dim_{\mathbb R}(N'\cap M_1)_{\mathrm{sa}}.
\]
Bakshi--Das--Liu--Ren proved that the number of minimal intermediate
subfactors is at most the kissing number $\tau_D$
\cite[Theorem~4.1]{BakshiDasLiuRen}.  Their elementary estimate
$\tau_D<3^D$, iterated over successive subintervals, gives the
whole-lattice bound
\begin{equation}\label{eq:bdlr-general}
 |\mathcal L(N\subset M)|\leq9^I
 \qquad\text{\cite[Theorem~4.6]{BakshiDasLiuRen}}.
\end{equation}
The bound $\tau_D\leq(\sqrt6)^D$ in
\cite{PalcouxSqrtSix}, whose Euclidean-volume proof is formalized in
Lean~4/Mathlib and registered in Palomar, gives
$|\mathcal L(N\subset M)|\leq6^I$ by the same recursion.

The quadratic exchange bound instead counts the whole lattice directly:
\begin{equation}\label{eq:comparison-bounds}
 |\mathcal L(N\subset M)|
 \leq2^{D-1}
 \leq2^{\lfloor I\rfloor-1}
 <6^I.
\end{equation}
Thus it improves these elementary index-only bounds and is sensitive to the
second relative commutant.  The two methods nevertheless capture different
information: angular separation controls minimal intermediates, while
local rigidity isolates all exchange idempotents at once.

For asymptotic comparison, the Kabatiansky--Levenshtein bound
\cite{KabatianskyLevenshtein}, combined with the larger-angle comparison
for spherical codes \cite[\S{}2, (2.5)--(2.6) and footnote~1]{CohnZhao},
gives $\tau_D\leq2^{(0.401+o(1))D}$.  The angle recursion therefore yields
$2^{(0.802+o(1))I}$, asymptotically sharper than the index-only
specialization of \eqref{eq:subfactor-bound-chain}.  The local
bound remains direct and can be much smaller when $D\ll I$.  It may be
possible to improve it further by exploiting the sparse or
multihomogeneous structure of the exchange equations.  When
$N'\cap M_1$ is abelian, the quasi-polynomial whole-lattice estimate of
\cite{BakshiDasLiuRen} is asymptotically stronger than these
index-only exponential estimates.  The angle and local
rigidity methods are therefore complementary.
\section{Examples and applications}\label{sec:examples}

The examples below separate the local mechanism from the additional
hypotheses used in the applications.

\begin{example}[Vector spaces]
In $\Vect_{\kk}$, separability is equivalent to semisimplicity
\cite[Chapter~II, \S2]{DeMeyerIngraham}, and the Wedderburn--Artin theorem
\cite[Chapter~1, \S\S2--3]{Lam} classifies finite-dimensional semisimple
$\kk$-algebras as finite direct sums of full matrix algebras over $\kk$.  Thus,
for fixed $d=\dim_{\kk}A$, there are only finitely many isomorphism classes.
\end{example}

\begin{example}[Why inner conjugacy is necessary]
For $A=\kk^2$ and $X=M_2(\kk)$, the embeddings $A\to X$ form an
infinite algebraic family of rank-one idempotents, but a single
$\mathrm{GL}_2(\kk)$-conjugacy class.  Hence inner conjugacy is the correct
general conclusion for a nonconnected ambient algebra.
\end{example}

\begin{application}[Rational representations of linearly reductive groups]
Let $G$ be a linearly reductive algebraic group over the algebraically closed
field $\kk$, and let $X$ be a finite-dimensional rational $G$-algebra.  If
$X^G=\kk 1$, then $X$ has only
finitely many $G$-stable separable unital subalgebras.  More generally, for a
fixed rational $G$-module $B$, there are only finitely many equivariant
separable algebra structures on $B$ up to equivariant algebra isomorphism.
For a $G$-stable algebra, linear reductivity allows an ordinary
separability section to be chosen $G$-equivariantly.  The subalgebra
assertion therefore follows from Corollary~\ref{cor:semisimple}, whereas
the fixed-module assertion follows from Theorem~\ref{thm:fixed-object},
both applied in $\Rep(G)$.
\end{application}

The same corollary applies to every semisimple tensor category.  In
particular, the ambient category need not be finite; that is, it need not be
a fusion category.

\begin{application}[Bimodule quantum channels]
Let $N\subseteq M$ be an irreducible finite-index inclusion of
$\mathrm{II}_1$ factors.  Two normal states $\varphi,\psi$ on $M$ are
called equivalent if $\varphi\circ\Phi=\psi$ and
$\psi\circ\Psi=\varphi$ for normal unital completely positive
$N$--$N$-bimodule maps $\Phi,\Psi:M\to M$.
Huang--Jiang--Liu--Wu \cite[Theorem~1.3]{HuangJiangLiuWuStates}
use the intermediate-subfactor bound of
Theorem~\ref{thm:subfactor-bound}, together with quantum Fourier analysis,
to obtain
\[
 |[\varphi]|
 \leq [M:N]\,2^{\dim_{\CC}(N'\cap M_1)-1}
 \leq [M:N]\,2^{\lfloor[M:N]\rfloor-1}.
\]
This bounds the number of states in each equivalence class, rather than
the number of equivalence classes.
\end{application}
\Needspace{8\baselineskip}
\section{Scope, sharpness, and open directions}\label{sec:scope}

\subsection{Comparison with earlier finiteness and lattice results}
\label{subsec:comparison-methods}

The framework of \cite{GhoshPalcouxExchange,GhoshPalcouxLattices}
develops ambient duals and selfduality criteria for Frobenius subalgebras,
exchange relations, a linear-monoidal Landau theorem, coherent pairs and
sublattices, rigidity invariance, and formal-angle invariants.  Its lattice
criteria are distinct from its finiteness argument, which uses lattice
and coherence hypotheses together with positivity.

Here split embeddings and local rigidity give qualitative finiteness,
and the exchange equations yield intrinsic bounds, without lattice or
coherence assumptions.  Connectedness gives a finite poset, without
asserting that it is a lattice.  The approaches are complementary and
share the exchange framework; the operator-algebraic bounds are compared
in \S\ref{subsec:subfactor-comparison}.

\subsection{Where each assumption enters}

Hom-finiteness supplies finite-dimensional parameter spaces, and
separability contracts Hochschild cohomology in degrees one and two.
Neither semisimplicity nor unitarity enters local rigidity.
Connectedness makes inner conjugation trivial; degree estimates make
the resulting finiteness statements effective.

Theorem~\ref{thm:subobjects} needs finitely many underlying source
isomorphism classes.  This holds for all subobjects in a semisimple
category, and for split subobjects by
Lemma~\ref{lem:split-source-types}, without abelianity or finite length.
Frobenius embeddings split by definition, so their qualitative
finiteness follows from Corollary~\ref{cor:split-subobjects}; the exchange
locus supplies the sharper intrinsic bound.
For connected Frobenius algebras, nonzero pivotal dimension implies
separability (Lemma~\ref{lem:pivotal-frobenius-separable}).  Positivity
ensures this nonvanishing and enters the dimension estimates, while
unitarity supplies the compatible structures in the operator-algebraic
applications.

\subsection{The limits of local rigidity}

An open separable orbit need not be closed.  On $\kk1\oplus\kk x$, the
polynomial family $x^2=tx$ gives $\kk\times\kk$ for $t\ne0$ and the
nonseparable dual-number algebra at $t=0$.  This specialization is a
degeneration of the multiplication table.
Separability of the source is also essential: the maps
$\kk[\zeta]/(\zeta^2)\to\kk[u]/(u^2)$, $\zeta\mapsto au$, form an
affine line of distinct inner-conjugacy classes, since the codomain is
commutative.

\subsection{The nonsemisimple ambient case}\label{subsec:nonsemisimple}

Finite length does not ensure finitely many subobject isomorphism
classes.  For the additive algebraic group
$\mathbb G_{\mathrm a}=(\mathbb A^1_{\kk},+)$, the
representation of $\mathbb G_{\mathrm a}^2$ on
$V=\kk e_0\oplus\kk e_1\oplus\kk e_2$ given by
\[
 \rho(s,t)e_0=e_0,\qquad
 \rho(s,t)e_1=e_1+s e_0,\qquad
 \rho(s,t)e_2=e_2+t e_0
\]
has pairwise nonisomorphic subobjects
$\kk e_0\oplus\kk(\alpha e_1+\beta e_2)$ indexed by
$[\alpha:\beta]\in\mathbb P^1(\kk)$.
Nonsplit separable algebra subobjects can likewise form an infinite
family, even inside a connected Frobenius algebra in a finite symmetric
tensor category over $\CC$.  Thus the finite-source or split hypothesis
cannot simply be omitted; compatible Frobenius-subalgebra embeddings
satisfy the latter automatically.

\subsection{Sources of loss and possible refinements}

B\'ezout counts the entire parameter variety, including irrelevant
nonseparable components, and ignores sparsity, dependencies, stabilizers,
and categorical structure.  Possible refinements include rank strata,
multihomogeneous estimates, and sparse methods such as
\cite{Bernshtein}; the latter also require control of zero coordinates
and positive-dimensional components.

Unitarity makes the relevant projections orthogonal, but
$\dim_{\mathbb R}\End_{\C}(X)_{\mathrm{sa}}
=\dim_{\CC}\End_{\C}(X)$, so self-adjoint coordinates alone do not
halve the exponent.  A stronger estimate must exploit the exchange
equations on projection strata or combine them with angle bounds.
Likewise, local rigidity supplies no asymptotic counting rate beyond
$N_F^{\mathrm{inn}}(X)\le2^{r_X}$; improving the base requires uniform
information about these quadratic systems
(Remark~\ref{rem:why-base-two}).  In the subfactor case, the defect
formula \eqref{eq:subfactor-defect} measures the loss in replacing the
relative-commutant dimension by the index; comparisons with angle bounds
must use the same parameter (\S\ref{subsec:subfactor-comparison}).

\subsection{Further questions}

\begin{enumerate}[label=\textup{\arabic*.}]
  \item Which objects in nonsemisimple tensor categories have finitely
  many subobject isomorphism classes?

  \item \label{q:optimal-bounds}
  In positive tensor categories, especially integral fusion categories,
  can the bound $2^{r_X}$ be replaced asymptotically by $c^{r_X}$ with
  $c<2$, or by a subexponential bound?  How do optimal exchange and angle
  bounds compare for irreducible subfactors?

  \item For nonsplit families satisfying the finite-source criterion of
  Theorem~\ref{thm:subobjects}, which additional hypotheses give effective
  bounds controlled by the ambient endomorphism space?
\end{enumerate}

\Needspace{12\baselineskip}
\section*{Acknowledgments and history}
The authors thank Zhengwei Liu and Sergey Neshveyev for discussions of the
historical status of the unitary local-rigidity statement.
They also thank Keshab Chandra Bakshi and Sebastian Burciu for their
interest in this work and for drawing their attention to additional
relevant references.
Mainak Ghosh thanks the
Beijing Institute of Mathematical Sciences and Applications for its
invitation and hospitality.

Discussions of Ocneanu rigidity began in January 2025, and the present
research direction took shape in April 2026.  The work intensified in
August 2026 around the International Congress of Basic Science.
GPT-5.6 Sol and GPT-6 Astra were used as research assistants.  The authors
take full responsibility for the final manuscript.

\end{document}